\documentclass[12pt,a4paper]{article}
\usepackage{amsmath}
\usepackage{amssymb}
\usepackage{cases}
\usepackage{bbm}
\usepackage{mathrsfs}
\usepackage{graphicx}
\usepackage{amsfonts}

\usepackage{verbatim}
\usepackage{enumerate}
\usepackage{theorem}
\usepackage{color}
\usepackage[colorlinks, linkcolor=red, citecolor=blue]{hyperref}
\usepackage{cite}
\makeatletter
\def\tank#1{\protected@xdef\@thanks{\@thanks
		\protect\footnotetext[0]{#1}}}
\def\bigfoot{
	
	\@footnotetext}
\makeatother

\newcommand{\ea}{\end{array}}
\allowdisplaybreaks

\newtheorem{theorem}{Theorem}[section]
\newtheorem{lemma}{Lemma}[section]

\newtheorem{proposition}[theorem]{Proposition}

\newtheorem{remark}{Remark}[section]

\newtheorem{definition}[theorem]{Definition}

\def\<{{\langle}}
\def\>{{\rangle}}

{\theorembodyfont{\rmfamily}
}

\title{Strong well-posedness of the two-dimensional stochastic Navier-Stokes equation on moving domains}

\author{{Ping Chen}$^1$\footnote{E-mail:2025800007@hfut.edu.cn}~~~{Tianyi Pan}$^2$\footnote{E-mail:ptyzjg@gmail.com}~~~ {Tusheng Zhang}$^{2,3}$\footnote{E-mail:tusheng.zhang@manchester.ac.uk}
\\
\small  1. School of Mathematics,
\small  Hefei University of Technology,\\
\small  Hefei, Anhui 230009, China.\\
\small  2. School of Mathematical Sciences,
\small  University of Science and Technology of China,\\
\small  Hefei, Anhui 230026, China.\\
\small  3. Department of Mathematics, University of Manchester,\\
\small  Oxford Road, Manchester, M13 9PL, UK.
}
\date{}
\newenvironment{proof}{\par\noindent{\bf Proof:}}{\hspace*{\fill}$\blacksquare$\par}

\begin{document}
\maketitle
\noindent \textbf{Abstract:}
In this paper, we establish the strong ($H^1$) well-posedness of the  two dimensional stochastic Navier-Stokes equation with multiplicative noise on moving domains. Due to the nonlocality effect, this equation exhibits a ``piecewise" variational setting. Namely the global well-posedness of this equation is {obtained through} the well-posedness of a family of stochastic partial differential equations (SPDEs) in the variational setting on small time-{intervals}. We first {consider} the well-posedness on each {small} time interval, which does not have (nonhomogeneous) coercivity. Subsequently, we give an estimate of {the} lower bound of {the} length of the time-interval, which enables us to obtain the global well-posedness.


\vspace{4mm}


\vspace{3mm}
\numberwithin{equation}{section}
\tableofcontents
\section{Introduction}
In this paper, we establish the well-posedness of the two dimensional stochastic Navier-Stokes equation on moving domains $\{\mathcal{O}_t\}_{t\in[0,T]} $, which is written as follows
\vskip -0.5cm
\begin{numcases}{}\label{ns}
	\nonumber\dot{u}(t,x)+(u\cdot\nabla) u(t,x)+\nabla p=\Delta u(t,x)+\sigma\big(t,u(t)\big)\dot{W}, \quad (t,x)\in \mathcal{D}_T,\\
	\nonumber\mathrm{div}\ u(t,x)=0, \quad (t,x)\in \mathcal{D}_T,\\
	\nonumber u(t,x)=0, \ \ \ \ \ \ \ \ (t,x)\in \partial\mathcal{D}_T,\\
	u(0,y)=u_0(y), \ \ \ \ \ \ \   y\in \mathcal{O}_0.
\end{numcases}
The noncylindrical domain $\mathcal{D}_T\subseteq [0,T]\times\mathbb{R}^{2}$ and its boundary $\partial\mathcal{D}_T$ are defined as
\vskip -0.6cm
\begin{align}\label{09022055}
	\mathcal{D}_T=\mathop{\scalebox{1.5}[1.5]{$\cup$}}_{t\in[0,T]} \{t\}\times\mathcal{O}_t,\  	\partial\mathcal{D}_T=\mathop{\scalebox{1.5}[1.5]{$\cup$}}_{t\in[0,T]} \{t\}\times\partial\mathcal{O}_t.
\end{align}
In this equation, $u:\mathcal{D}_T\rightarrow\mathbb{R}^2$ represents the unknown velocity of an incompressible viscous fluid, $p$ denotes the pressure and the stochastic force $\sigma(t,u)dW$ represents the influence of internal, external or environmental noises.\par
In the real world, the evolution of objects is often accompanied with  the movement of the boundary of the occupied region. Moving boundary problems are ubiquitous in nature, technology and engineering. The Navier-Stokes equation on moving domains is one of the most important models in this field and has numerous applications, e.g. in biomedical engineering, cardiovascular medicine and geophysics. We {refer} the readers to
\cite{C} for a recent survey.\par
The research on the deterministic counterpart ($\sigma=0$) of equation (\ref{ns}) has been ongoing for half a century. Regarding weak solutions, Inoue and Wakimoto \cite{IW}, as well as Miyakawa and Teramoto \cite{MT}, {obtained} the existence of a weak solution  on bounded moving domains. Salvi \cite{R} {established} the existence of a weak solution on moving exterior domains. Suitable weak solutions are also constructed by Choe, Jang and Yang \cite{CJY}. A recent study in \cite{CFLS} focused
on the pullback attractors of the weak solutions. \par
Based on the maximal regularity theory and the (nonhomogeneous) semigroup theory, there has also been some important results regarding the $L^p$ solution to the deterministic stochastic Navier-Stokes equation on moving domains. Saal \cite{S} investigated the unique strong  local $W^{1,p}$ solution. Farwig, Kozono, Tsuda and Wegmann \cite{FKTW} constructed the unique $W^{1,p}$ periodic solution for (\ref{ns}) on slightly-deformed bounded moving domains.  Farwig and Tsuda \cite{FT} constructed the unique $L^p$ solution on slightly-deformed half space within a certain function class.\par
However, in the stochastic case, difficult issues arise. We now highlight the challenges we face in the stochastic case. \par
Firstly, when considering the weak solution, although the well-posedness for two dimensional case has been established in some papers, such as \cite{MT} and \cite{SSY}, their proofs on the uniqueness are informal. {We also mention} that in our previous work \cite{PWZZ}, we formulate a general setting, the so called nonhomogeneous  monotonicity, under which the uniqueness of SPDEs on moving domains can be established. Unfortunately, some of the conditions there is not satisfied by the stochastic Navier-Stokes equations.
Intuitively, the obstacle stems from the nonlocal trait of this equation.   The nonlocality makes the deformation of domain a ``singular perturbation" for weak solutions. A similar difficulty was mentioned in \cite{DGH}. Also, it is for the same reason that in \cite{DMM}, a rigid restriction on the movement of the domain is posed. \par
Among those methods used to establish $L^p$ solutions of deterministic Navier-Stokes equations on moving domians, pathwise arguments and time-differentibility are frequently involved, suggesting that the techniques {are not} suitable for the stochastic case. Additionally, with their methods, smallness of the external force should be required, which is not reasonable in stochastic settings. \par
Instead, {in this paper}, we {consider} the strong ($H^1$) solution of the stochastic Navier-Stokes equation (\ref{ns}). We prove the well-posedness of the strong solution to the two dimensional stochastic Navier-Stokes equation on moving domains $\mathcal{D}_T$. Although nonlocality {is still the main} obstacle, this challenge {can be overcomed in the $H^1$-setting}. Moreover, in our paper, no smallness assumptions are made either on the force or on the deformation of the domain.\par
In {approach}, we first apply the Piola transformation, which has been previously used in many papers such as \cite{IW,MT,FKTW,FKW}, to transform the original equation into an equation on a fixed domain. Through a rigorous proof concerning the well-posedness of a family of generalized steady Stokes problems, we {show} that the stochastic Navier-Stokes equation (\ref{ns}) exhibits a ``piecewise" variational setting. That is, the well-posedness of (\ref{ns}) on $[0,T]$ is decomposed into the well-posedness of a family of SPDEs considered on small sub-intervals $\{((i-1)\delta,i\delta\wedge T]\}_{i\geq 1}$ under different Gelfand triplets $V_i\dot{=}H^2\cap H^1_{0,\sigma}(\mathcal{O}_{{(i-1)\delta}})\hookrightarrow H_i\dot{=}H^1_{0,\sigma}(\mathcal{O}_{{(i-1)\delta}})\hookrightarrow V_i^*\dot{=}L^2_{\sigma}(\mathcal{O}_{{(i-1)\delta}})$, $i\leq N\dot{=}\lceil{\frac{T}{\delta}}\rceil$, written as follows
\vskip -0.5cm
\begin{numcases}{}\label{spde1}
	dX^i_t=A^i(t,X^i_t)dt+ B^i(t,X^i_t)dW^i_t, \quad t\in ({(i-1)\delta},{i\delta}\wedge T] ,\\
	X^i_{{(i-1)\delta}}=G_{i}X_{{(i-1)\delta}}^{i-1}\in H_i.\nonumber
\end{numcases}
In (\ref{spde1}), $G_i:H_{i-1}\rightarrow H_i,\ i\leq N$ is a family of bounded operators,
$$A^i(\cdot,\cdot):((i-1)\delta,i\delta\wedge T]\times V_i\rightarrow V_i^*, B^i:((i-1)\delta,i\delta\wedge T]\times V_i\rightarrow\ L_2(U,H_i),$$
are a family of measurable maps.\par
To solve {the equation \eqref{spde1} on each fragment $({(i-1)\delta},i\delta\wedge T]$,
 we first consider the globally modified Navier-Stokes equation, where $g_N(r)=1\wedge \frac{N}{r}$,
\vskip -0.5cm
\begin{numcases}{}\label{gns}
	\nonumber\dot{u}(t,x)+g_N(\|\nabla u\|_{L^2(\mathcal{O}_t)})(u\cdot\nabla) u(t,x)+\nabla p=\Delta u(t,x)+\sigma\big(t,u(t)\big)\dot{W},\ (t,x)\in \mathcal{D}_T,\\
	\nonumber\mathrm{div}\ u(t,x)=0,\ (t,x)\in \mathcal{D}_T,\\
	\nonumber u(t,x)=0,\ \ \ \ \ \ (t,x)\in \partial\mathcal{D}_T,\\
	u(0,y)=u_0(y), \ \ \ \ \ y\in \mathcal{O}_0.
\end{numcases}
We find that the transformed version of the modified equation (\ref{gns}) on each time interval $((i-1)\delta,i\delta\wedge T]$ under the Gelfand triplet $H^2\cap H^1_{0,\sigma}(\mathcal{O}_{{(i-1)\delta}})\hookrightarrow H^1_{0,\sigma}(\mathcal{O}_{{(i-1)\delta}})\hookrightarrow L^2_{\sigma}(\mathcal{O}_{{(i-1)\delta}})$ satisfies the framework in \cite{PWZZ}, which gives the well-posedness. To obtain the well-posedness of the equation (\ref{spde1}) on each time interval $((i-1)\delta,i\delta\wedge T]$, inspired by the estimate in \cite{KV}, we establish a uniform in $N$ moment estimate for $\Theta(\|\nabla u(t)\|_{L^2(\mathcal{O}_t)}^2)$ where $\Theta(x)=\mathrm{log}(1+\mathrm{log}(1+x))$. Based on this uniform moment estimate, we can let $N$ tend to $\infty$, proving the well-posedness of the equation  (\ref{spde1}). Finally, {piecing the equation \eqref{spde1} on small intervals together}, we successfully establish the global existence of  the stochastic Navier-Stokes equation (\ref{ns}).
\par
Concerning the uniqueness, unlike the case for the weak solution, the equation (\ref{spde1}) hold on $V_i^*=L^2_\sigma(\mathcal{O}_{(i-1)\delta})$. We can then use the generalized It\^{o}'s formula on $L^2_\sigma$ established in \cite{PWZZ} to deduce the uniqueness on each time interval $((i-1)\delta,i\delta\wedge T]$, leading to the uniqueness on $[0,T]$.\par
It's worth mentioning that the selection of $\delta$ depends on a uniform $L^2$ estimate for a family of generalized steady Stokes problems. To the best of our knowledge, the existing results on the well-posedness of generalized steady Stokes equations all have strict assumptions on the elliptic operators, which {can not be applied to} the elliptic operators {arising} in this paper (see e.g. \cite{BS,M} and references therein). To deal with this issue, we use the operator perturbation theory to show that, on a small time interval $[0,t_0]$, the generalized Stokes operators are close to the original Stokes operator, which has an $L^2$ estimate. Then we replace the reference {domain} $\mathcal{O}_{0}$ by $\mathcal{O}_{t_0}$ and repeat the previous argument, proving the uniform $L^2$ estimate on $[
t_0,t_0+t_1]$. The existence of a uniform lower bound for $\{t_i\}_i$ is guaranteed by a combination of the uniform $L^2$ estimate for steady Stokes equations on $\{\mathcal{O}_t\}_{t\in[0,T]}$ established in \cite{H} and a uniform in time estimate on the continuity of the coefficients of the elliptic operators {involved}. This lower bound serves as $\delta$ in the above-mentioned paragraphs.\par
{The remaining part of this paper is organized as follows.}
 In Section 2, we provide preliminaries regarding the function spaces and the Piola transformation of each term in equation (\ref{ns}). In Section 3, we introduce the main result of this paper. The conditions (\hyperlink{C1}{C1})-(\hyperlink{C4}{C4}) proposed in \cite{PWZZ} are verified on a small time interval $[0,t_0]$ in Section 4.1. These conditions play an important role when we establish the equivalence between modified equation (\ref{gns}) and its transformed equation  on $[0,t_0]\times\mathcal{O}_0$ in Section 4.2. In Section 4.3, we  prove the well-posedness of the modified equation (\ref{gns}) on the small time interval $[0,t_0]$. In Section 5, we provide a uniform moment estimate to obtain the well-posedness of the stochastic Navier-Stokes equation (\ref{ns}) on the small interval $[0,t_0]$. In Section 6, we give a lower bound of the length of the interval $[0,t_0]$ when different $\mathcal{O}_t$, with $t\in[0,T]$,  are chosen as the reference domain, thus proving the well-posedness on $[0,T]$.
\vskip 0.3cm
Here are some conventions used throughout this paper:

(i) We let $C$, $c$, etc. represent {generic} positive constants whose value may change from line to line. Other constants will be denoted as $c_0$, $c_1$, \dots. The dependence of constants on parameters, if necessary, will be indicated, e.g. $C_N$.

(ii) We denote $\mathcal{L}(X,Y)$ to be the space of bounded linear operators from $X$ to $Y$, with its norm denoted by $\|\cdot \|_{\mathcal{L}(X,Y)}$. If $X = Y$, we use $\mathcal{L}(X)$ instead of $\mathcal{L}(X,Y)$, with its norm denoted by $\| \cdot \|_{\mathcal{L}(X)}$.

(iii) We denote by $\mathcal{L}_2(U,H)$ the space of Hilbert-Schmidt operators from a separable Hilbert space $U$ to another separable Hilbert space $H$, equipped with the usual Hilbert-Schmidt norm $\|\cdot\|_{L_2(U,H)}$.

(iv) We use Einstein summation convention, i.e. take sum over repeated indices.

(v) If the time dependence of certain quantities is clear from the context, we suppress the $t$-notation for brevity, unless it needs to be explicitly emphasized.

\section{Preliminaries}\label{sec:preliminaries}

Let $\{\mathcal{O}_t\}_{t\in[0,T]}$ be a family of $C^3$ bounded open domains in $\mathbb{R}^2$, indexed by $t \in [0,T]$. We assume that there exists a map $r(\cdot, \cdot): [0,T] \times {\overline{\mathcal{O}}_0} \rightarrow \mathbb{R}^2 $ such that
\begin{itemize}
	\item [(i)] For any $t \in [0, T]$, $r(t, \cdot): \mathcal{O}_0 \rightarrow \mathcal{O}_t$ is a $C^3$ diffeomorphism between $\mathcal{O}_0$ and $\mathcal{O}_t$.
	\item [(ii)] $r \in C^{2,3}_{t,x}\big([0, T] \times \overline{\mathcal{O}}_0, \mathbb{R}^2 \big)$, that is, for any $k \leq 2$, $|\alpha|\leq 3$,
	\[ \frac{\partial^{k+|\alpha|} }{\partial t^k \partial x^\alpha}r(\cdot,\cdot)\in C([0,T]\times\overline{\mathcal{O}}_0,\mathbb{R}^2). \]
\end{itemize}
Furthermore, we denote the inverse map of $r(t, \cdot)$ by $\bar{r}(t, \cdot): \mathcal{O}_t \rightarrow \mathcal{O}_0$. As in \cite{MT,FKW}, we assume that
\begin{equation}\label{M}
	\mathrm{det}\Big(\big\{\frac{\partial \bar{r}_i}{\partial x_j}(t,x)\big\}_{i,j\leq 2}\Big)=1, \quad (t, x) \in \mathcal{D}_T,
\end{equation}
{meaning that the transformation $r(t, \cdot)$ preserves the volume of the domains.}

Since we will transform the stochastic Navier-Stokes equation \eqref{ns} into an equation on a fixed domain,
for a function $u: \mathcal{D}_T \rightarrow \mathbb{R}^2$, we introduce the so-called Piola transformation:
\begin{align}\label{01270222}
	\widetilde{u}_i(t,y)=\frac{\partial \bar{r}_i}{\partial x_j}\big(t,r(t,y)\big)u_j(t,r(t,y)):[0,T]\times\mathcal{O}_0\rightarrow \mathbb{R}^2.
\end{align}
{If $\mathrm{div}_x\ u(t,x)=0$}, then for the $\mathbb{R}^2$-valued vector field $\widetilde{u}$, we have
\begin{align*}
	\mathrm{div}_y\widetilde{u}(y)=&	\frac{\partial \widetilde{u}_i}{\partial y_i}(t,y)=\frac{\partial}{\partial y_i}\big[\frac{\partial \bar{r}_i}{\partial x_j}(t,r(t,y))u_j(t,r(t,y))\big]\\=& \frac{\partial^2 \bar{r}_i}{\partial x_j\partial x_k}\big(t,r(t,y)\big)\frac{\partial r_k}{\partial y_i}(t,y)\cdot u_j\big(t,r(t,y)\big)+\frac{\partial \bar{r}_i}{\partial x_j}\big(t,r(t,y)\big)\frac{\partial u_j}{\partial x_k}\big(t,r(t,y)\big)\frac{\partial r_k}{\partial y_i}(t,y)\\=&\frac{\partial u_k}{\partial x_k}\big(t,r(t,y)\big)=\mathrm{div}_x\ u(t,r(t,y))=0,
\end{align*}
where the first term on the right-hand side of the third equality disappears due to \eqref{M}.

Below, we apply the Piola transformation to each term in equation \eqref{ns}. For the time derivative term, we have
\begin{align}\label{01290152}
	\frac{\partial u_i}{\partial t}(t,x)=&\frac{\partial \widetilde{u}_j}{\partial t}(t,\bar{r}(t,x))\frac{\partial r_i}{\partial y_j}(t,\bar{r}(t,x))\ +\frac{\partial \widetilde{u}_j}{\partial y_k}(t,\bar{r}(t,x))\frac{\partial\bar{r}_k}{\partial t}(t,x)\frac{\partial r_i}{\partial y_j}(t,\bar{r}(t,x))
	\\ &\ +\widetilde{u}_j(t,\bar{r}(t,x))\frac{\partial^2 r_i}{\partial y_j\partial t}(t,\bar{r}(t,x))+\widetilde{u}_j(t,\bar{r}(t,x))\frac{\partial^2 r_i}{\partial y_j\partial y_k}(t,\bar{r}(t,x))\frac{\partial \bar{r}_k}{\partial t}(t,x).	\nonumber
\end{align}
We denote the Christoffel symbol by $\Gamma^i_{jk}(t,y)=\frac{\partial \bar{r}_i}{\partial x_l}(t,r(t,y))\frac{\partial^2 r_l}{\partial y_j\partial y_k}(t,y)$, and denote the covariant derivative of $\widetilde{u}$ by
\begin{align}\label{01270740}
	(\nabla_j \widetilde{u})_i(y)=\frac{\partial \widetilde{u}_i}{\partial y_j}+\Gamma^i_{jk}(t,y)\widetilde{u}_k(t,y).
\end{align}
Define
\begin{align}\label{02061619}
	h_{jk}(t,y)=\frac{\partial r_i}{\partial y_j}(t,y)\frac{\partial r_i}{\partial y_k}(t,y),\ \ h^{jk}(t,y)=\frac{\partial \bar{r}_k}{\partial x_i}\frac{\partial \bar{r}_j}{\partial x_i}(t,r(t,y)).
\end{align}
Then, using {a similar computation leading to} \eqref{01290152}, we obtain
\begin{equation}\label{1}
	\frac{\partial u_i}{\partial x_l}(t, r(t,y))=\frac{\partial r_i}{\partial y_j}\frac{\partial \bar{r}_k}{\partial x_l}(\nabla_k \tilde{u})_j,
\end{equation}
\begin{equation}\label{2}
	\Delta u_p(t, r(t,y))=\frac{\partial}{\partial y_k}\Big(\frac{\partial}{\partial y_j}(\widetilde{u}_l\frac{\partial r_p}{\partial y_l})h^{jk}\Big).	
\end{equation}
By \eqref{01270222} and \eqref{1}, we get
\begin{align}\label{01291431}
	\left [\frac{\partial \bar{r}_i}{\partial x_p}(u_j\cdot\nabla_ju)_p\right ](t, r(t,y))=\widetilde{u}_j({\nabla_j \widetilde{u}})_i\ \dot{=}\ N(\widetilde{u},\widetilde{u})^i.
\end{align}

Next, we define the following operator: for any $v:\mathcal{O}_0\rightarrow \mathbb{R}^2$,
\begin{align}\label{3}
	(L_hv)_i=\frac{\partial \bar{r}_i}{\partial x_p}(t, r(t,y))\frac{\partial}{\partial y_k}\Big(\frac{\partial }{\partial y_j}(v_l\frac{\partial r_p}{\partial y_l})h^{jk}\Big).
\end{align}
Note that if we denote the inversion of Piola transformation on $v: [0,T]\times\mathcal{O}_0\rightarrow \mathbb{R}^2$ by
\begin{align}\label{01270556}
	\overline{v}_i(t,x)=\frac{\partial {r}_i}{\partial y_j}\big(t,\bar{r}(t,x)\big)v_j(t,\bar{r}(t,x)):\mathcal{D}_T\rightarrow \mathbb{R}^2,
\end{align}
then for every $w\in[H^2(\mathcal{O}_0)\cap H_0^1(\mathcal{O}_0)]^2$, we have
\begin{align}\label{01270721}
	\int_{\mathcal{O}_0} (L_hv)_kw_lh_{kl}dy=-\int_{\mathcal{O}_0} h^{kl}h_{ij}
	\big(\nabla_k v\big)_i (\nabla_l w)_j dy=-\int_{\mathcal{O}_t}\nabla \overline{v}(x)\cdot\nabla\overline{w}(x) dx.
\end{align}

Now, we introduce some function spaces and notations.

Given a $C^3$ bounded open domain $\mathcal{O}$, denote by $L_\sigma^2(\mathcal{O})$ the space of $\mathbb{R}^2$-valued square integrable, divergence-free functions on $\mathcal{O}$, which is equipped with the norm
\begin{align*}
	\|u\|^2_{L^2(\mathcal{O})}=\sum_{i=1}^{2}\int_{\mathcal{O}} |u_i(x)|^2dx.
\end{align*}

Let $H_{0,\sigma}^1(\mathcal{O})$ denote the closed subspace of $\big[H_0^1(\mathcal{O})\big]^2$, containing all divergence-free functions, equipped with the norm
\begin{align*}
	\|u\|^2_{H^1(\mathcal{O})}=\sum_{i=1}^{2}\int_{\mathcal{O}} |u_i(x)|^2dx+\sum_{i=1}^{2}\int_{{\mathcal{O}}}|\nabla u_i(x)|^2dx.
\end{align*}
By the Poincar\'e inequality, the $\|\cdot\|_{H^1(\mathcal{O})}$ norm is equivalent to the norm $\|\cdot\|_{1,0}$, defined by
$$\|u\|^2_{1,0}=\int_{{\mathcal{O}}}|\nabla u(x)|^2 dx.$$

Let $H^{-1}_\sigma(\mathcal{O})$ be the dual space of $H_{0,\sigma}^1(\mathcal{O})$. The corresponding norm in $H^{-1}_\sigma(\mathcal{O})$ is given by
$$\|f\|_{H^{-1}_\sigma(\mathcal{O})}\dot{=}\sup_{v\in H_{0,\sigma}^1(\mathcal{O})}\frac{ { _{H_\sigma^{-1}(\mathcal{O})}\langle f,v\rangle_{H_{0,\sigma}^{1}(\mathcal{O})} }   }{\|v\|_{1,0}},$$
{where $ _{H_\sigma^{-1}(\mathcal{O})}\langle \cdot , \cdot \rangle_{H_{0,\sigma}^{1}(\mathcal{O})}$ represents the dual pair between $H^{-1}_\sigma(\mathcal{O})$ and $H_{0,\sigma}^1(\mathcal{O})$.}

Similarly, we denote by $H^2\cap H^1_{0,\sigma}(\mathcal{O})$ the subspace of $\big[ H^2\cap H_0^1(\mathcal{O})\big]^2$, containing all divergence-free functions, equipped with the norm
\begin{align*}
	\|u\|^2_{H^2(\mathcal{O})}=\int_{\mathcal{O}} |u(x)|^2dx+\int_{\mathcal{O}} |\nabla u(x)|^2dx+\int_{\mathcal{O}} |D^2 u(x)|^2dx.
\end{align*}
Define $A=P\Delta$ as the Stokes operator on $\mathcal{O}$, where $P$ denotes the Leray projection. Using the operator $A$, we can also equip the space $\big[H^2\cap H_{0,\sigma}^1(\mathcal{O})\big]^2$ with the norm
$$\|u\|_5=\|Au\|_{L^2(\mathcal{O})},$$
which is equivalent to the norm $\|\cdot\|_{H^2(\mathcal{O})}$ by Proposition \ref{estimate3} in the Appendix.
For $f\in L_\sigma^2(\mathcal{O}), v\in H^2\cap H^1_{0,\sigma}(\mathcal{O})$, the dual pair $\langle f,v\rangle$ under the Gelfand triplet $H^2\cap H_{0,\sigma}^1(\mathcal{O})\hookrightarrow H_{0,\sigma}^1(\mathcal{O})\hookrightarrow L^2_\sigma(\mathcal{O})$ is defined as
$$\langle f ,v\rangle=-(f,Av)_{L^2(\mathcal{O})}.$$

We will consider the above norms on the family of domains $\{\mathcal{O}_t\}_{t\in[0,T]}$. By Proposition \ref{estimate3} in the Appendix, we know that the norms $\| \cdot \|_{H^2(\mathcal{O}_t)}$ and $\| \cdot \|_5$ are equivalent. Using the Rayleigh-Faber-Krahn inequality (see e.g. \cite{K}), we also derive that the norm $\|\cdot\|_{H^1(\mathcal{O}_t)}$ is equivalent to the norm $\|\cdot\|_{1,0}$ on $H^1_{0,\sigma}(\mathcal{O}_t)$, and the equivalence constants for both pairs of norms are independent of $t \in [0,T]$. Therefore, we will use these two pairs of norms interchangeably in the later sections.

\section{Statement of the main result}

In this section, we introduce the precise definition of solutions and state the main result.  Recall the Piola transformation $\widetilde{\cdot}$ and its inverse $\overline{\cdot}$ given by \eqref{01270222} and \eqref{01270556}, respectively. Let us begin by introducing necessary function spaces on time-dependent domains. Indeed, the function spaces on non-cylindrical domains $\mathcal{D}_T$ are defined similarly to those in Subsection 3.3 in \cite{PWZZ}, we write down them here for the sake of {completeness}. Denote by ${C^2_{c,\sigma}([0,T] \times \mathcal{O}_0)}$ the space of all $\mathbb{R}^2$-valued $C^2$ functions on $[0,T] \times \mathcal{O}_0$ with compact support and zero divergence. We set
\begin{align*}
	C\big([0,T],H_{0,\sigma}^1(\mathcal{O}_\cdot)\big)\dot{=}& \big \{  \{ u(t) \} _{t\in[0,T]}: u(t) \in H_{0,\sigma}^1\big(\mathcal{O}_t) \  \mbox{for any} \ t \in [0,T]  \\ & \  \mbox{such that} \ \widetilde{u} \in  C\big([0,T], H_{0,\sigma}^1(\mathcal{O}_0)\big)   \big \} ,
\end{align*}

\vskip -1.0cm

\begin{align*}
	C\big( [0,T], L_{\sigma}^2(\mathcal{O}_\cdot) \big) \dot{=}& \big \{  \{ u(t) \} _{t\in[0,T]}: u(t) \in  L_{\sigma}^2(\mathcal{O}_t) \  \mbox{for any} \ t \in [0,T]  \\ & \  \mbox{such that} \ \widetilde{u} \in  C\big( [0,T], L_{\sigma}^2(\mathcal{O}_0) \big)   \big \} ,
\end{align*}

\vskip -1.0cm

\begin{align*}
	L^2\big([0,T],H^2\cap H_{0,\sigma}^1(\mathcal{O}_\cdot)\big) \dot{=}& \big \{  \{ u(t) \} _{t\in[0,T]}: u(t) \in  H^2\cap H_{0,\sigma}^1(\mathcal{O}_t) \  \mbox{for any} \ t \in [0,T]  \\ & \  \mbox{such that} \ \widetilde{u} \in  L^2\big([0,T],H^2\cap H_{0,\sigma}^1(\mathcal{O}_0)\big)   \big \} ,
\end{align*}

\vskip -1.0cm

\begin{align*}
	C^2_{c,\sigma}(\mathcal{D}_T) \dot{=}& \big \{ \varphi: \mathcal{D}_T \rightarrow \mathbb{R}^2 |\ \widetilde{\varphi} \in C^2_{c,\sigma}([0,T] \times \mathcal{O}_0)   \big \},
\end{align*}
where the corresponding norms on $C\big([0,T],H_{0,\sigma}^1(\mathcal{O}_\cdot)\big)$, $C\big( [0,T], L_{\sigma}^2(\mathcal{O}_\cdot) \big)$ and $L^2\big([0,T],H^2\cap H_{0,\sigma}^1(\mathcal{O}_\cdot)\big)$ are defined as $$\sup_{t \in [0,T]}\| \cdot \|_{H^1(\mathcal{O}_t)} ,  \sup_{t \in [0,T]}\| \cdot \|_{L^2(\mathcal{O}_t)}  \text{ and }  (\int_0^T \|\cdot \|^2_{H^2(\mathcal{O}_t)}dt)^{\frac{1}{2}}$$ respectively.

Let $(\Omega,\mathcal{F}, \{\mathcal{F}_{t}\}_{t\in [0,T]},\mathbb{P})$ be a filtered probability space where $\{\mathcal{F}_{t}\}_{t\in[0,T]}$ satisfies the usual conditions. We say {that} a family $\{ u(t) \}_{t \in [0,T]}$ {is} an $H^1_{0,\sigma}(\mathcal{O}_\cdot)$-valued process if for any $t \in [0,T]$, $u(t)$ is an $H^1_{0,\sigma}(\mathcal{O}_t)$-valued random variable. Concerning predictability, an $H^1_{0,\sigma}(\mathcal{O}_\cdot)$-valued random process $\{ u(t) \}_{t \in [0,T]}$ is said to be predictable with respect to the filtration $\{\mathcal{F}_{t}\}_{t\in [0,T]}$ if the transformed process $\{ \widetilde{u}(t) \}_{t \in [0,T]}$ is an $H^1_{0,\sigma}(\mathcal{O}_0)$-valued predictable process with respect to $\{\mathcal{F}_t\}_{t\in[0,T]}$.

Now we {introduce} the Wiener process $W$ and the diffusion coefficient $\sigma$ {appeared} in equation \eqref{ns}. We assume that the Wiener process $W$ is an $l^2$ cylindrical Wiener process, represented by a family of independent standard Wiener processes $\{\beta^k\}_{k\geq 1}$. The diffusion coefficient $\sigma(\cdot, \cdot)$ is represented as a family of mappings $\big\{ \{\sigma_k(t,\cdot):H_{0,\sigma}^1(\mathcal{O}_t)\rightarrow H_{0,\sigma}^1(\mathcal{O}_t)\}_{t \in [0,T]} \big\}_{k \geq 1}$. 
In addition, we assume that there exists $f\in L^1([0,T],\mathbb{R}_+)$ such that for a.e. $t \in [0,T]$, and for any $u, v \in H^1_{0,\sigma}(\mathcal{O}_t)$,
\begin{align}
	&\label{02062230-1}\sum_{k=1}^{\infty}\|\sigma_k(t,u)-\sigma_k(t,v)\|^2_{H^1(\mathcal{O}_t)}\leq f(t)\|u-v\|^2_{H^1(\mathcal{O}_t)},\\\label{02062230-2}&\sum_{k=1}^{\infty}\|\sigma_k(t,u)\|_{H^1(\mathcal{O}_t)}^2\leq f(t)\big(1+\|u\|^2_{H^1(\mathcal{O}_t)}\big),\\\label{02062230-3}&
	\sum_{k=1}^{\infty}\|\sigma_k(t,u)-\sigma_k(t,v)\|^2_{L^2_\sigma(\mathcal{O}_t)}\leq f(t)\|u-v\|^2_{L^2_\sigma(\mathcal{O}_t)}.
\end{align}

We now introduce the definition of solution for equation \eqref{ns}.
\begin{definition}\label{ws3}
	Let $u_0\in H_{0,\sigma}^1(\mathcal{O}_0)$. We say that an $H^1_{0,\sigma}(\mathcal{O}_\cdot)$-valued random process $\{ u(t)\}_{t \in [0,T]}$ is a strong solution to \eqref{ns} if
	\begin{itemize}
		\item[(i)]: $u\in C\big([0,T], {H}_{0,\sigma}^{1}\big(\mathcal{O}_\cdot)\big)\cap L^2\big([0,T],{H}^2\cap {H}^1_{0,\sigma}(\mathcal{O}_\cdot) \big)$                  $a.e.$ and
		$$ E\Big[\sup_{t\in[0,T]}\|u(t)\|^2_{{L}^{2}(\mathcal{O}_t)}+\int_{0}^{T}\|u(t)\|_{H^1(\mathcal{O}_t)}^2dt\Big]<\infty. $$
		\item[(ii)]: $u$ is an ${H}^{1}_{0,\sigma}({\mathcal{O}_\cdot})$-valued predictable random process with respect to the filtration $ \{ \mathcal{F}_t \} _{t\in[0,T]} $.
		\item[(iii)]: For any $\varphi\in C_{c,\sigma}^2({
			\mathcal{D}_T})$ and $t\in[0,T]$, we have $\mathbb{P}-a.s.$
		\begin{align}\label{ws2}
			&\nonumber \   \   \  \int_{\mathcal{O}_t}u(t,x)\cdot\varphi(t,x)dx-\int_{\mathcal{O}_0} u_0(y)\cdot\varphi(0,y)dy  \\ &\nonumber=\int_{0}^{t}\int_{{\mathcal{O}_s}} \Delta u(s,x)\cdot\varphi(s,x)dxds+\int_{0}^{t}\int_{{\mathcal{O}_s}}  u(s,x)\cdot{\varphi}'(s,x)dxds \\
			&\ \ + \int_{0}^{t}\int_{{\mathcal{O}_s}}(u\cdot\nabla \varphi)(s,x) u(s,x)dxds+ \sum_{k=1}^{\infty}\int_{0}^{t}\int_{\mathcal{O}_s} \sigma_k(s, u(s))\cdot\varphi(s,x)dxd\beta^k_s,
		\end{align}
		{where ${\varphi}'(s,x)$ stands for the derivative of $\varphi$ w.r.t. the time variable.}
	\end{itemize}
\end{definition}

\begin{remark}
	The definition of solution for equation \eqref{gns} is similar to that of equation \eqref{ns}, with \eqref{ws2} in Definition \ref{ws3} {replaced by}
	\begin{align}\label{ws2'}
		&\nonumber \   \   \  \int_{\mathcal{O}_t}u(t,x)\cdot\varphi(t,x)dx-\int_{\mathcal{O}_0} u_0(y)\cdot\varphi(0,y)dy  \\ &\nonumber=\int_{0}^{t}\int_{{\mathcal{O}_s}} \Delta u(s,x)\cdot\varphi(s,x)dxds+\int_{0}^{t}\int_{{\mathcal{O}_s}}  u(s,x)\cdot{\varphi}'(s,x)dxds \\
		&\ \ + \int_{0}^{t}\int_{{\mathcal{O}_s}}g_N(\|u(s)\|_{H^1(\mathcal{O}_s)})(u\cdot\nabla \varphi)(s,x) u(s,x)dxds+ \sum_{k=1}^{\infty}\int_{0}^{t}\int_{\mathcal{O}_s} \sigma_k(s, u(s))\cdot\varphi(s,x)dxd\beta^k_s.
	\end{align}
	
\end{remark}

We now state the main result in this paper.
\begin{theorem}\label{mainresult1}
	Assuming $u_0 \in H_{0,\sigma}^1(\mathcal{O}_0)$, {and that} the diffusion coefficient $\sigma$ satisfies \eqref{02062230-1}-\eqref{02062230-3}. Then the stochastic Navier-Stokes equation \eqref{ns} admits  a unique strong solution.
\end{theorem}
\begin{proof}
	The proof is divided into three parts, which are {respectively} given in the next three sections. We first show in Section \ref{section1} that for every $N>0$, the globally modified Navier-Stokes equation \eqref{gns} has a unique solution $u^N$ on a small interval $[0,t_0]$. Then, in Section \ref{section5}, by establishing a uniform in $N$ moment estimate for $\Theta(\|\nabla u^N(t)\|_{L^2(\mathcal{O}_t)}^2)$, where $\Theta(x)=\mathrm{log}(1+\mathrm{log}(1+x))$, we will {obtain} the well-posedness of \eqref{ns} on {the small time interval} $[0,t_0]$. In Section \ref{section6}, we {will} show that the Navier-Stokes equation \eqref{ns} is well-posed on $[0, T]$ by {piecing together the solutions on small time intervals}. This is done by giving a lower estimate of $t_0$ when choosing different $\mathcal{O}_t$, $t\in[0,T]$ as the reference domains.
\end{proof}

\section{Well-posedness of the modified equation on small time intervals}\label{section1}

Recall the globally modified Navier-Stokes equation given by \eqref{gns}. In this section, we will prove the following result.
\begin{proposition}\label{proposition4.1}
	Assuming that $u_0 \in H_{0,\sigma}^1(\mathcal{O}_0)$ and the diffusion coefficient $\sigma$ in equation \eqref{gns} satisfies \eqref{02062230-1}-\eqref{02062230-3}. Then there exists a small time $t_0 > 0$ such that for every $N>0$, the equation \eqref{gns} has a unique solution on $[0,t_0]$.
\end{proposition}
\begin{proof}
	We will apply Theorem 2.8 in a general setting in \cite{PWZZ}. The proof is again divided into three parts, which are given in the next three subsections. We first verify conditions (\hyperlink{C1}{C1})-(\hyperlink{C4}{C4}) as required in Section 2.2 of \cite{PWZZ}, in which $t_0>0$ is chosen to satisfy (\hyperlink{C4}{C4}). Then, through the Piola transformation we show that the existence and uniqueness of solutions to equation \eqref{gns} is equivalent to the existence and uniqueness of solutions to a SPDE of the form \eqref{spde1} on $[0,t_0]$. The proof of this correspondence is given in Subsection \ref{section3}. In Subsection \ref{section4}, we show that the SPDE \eqref{spde1} is well-posed by further verifying conditions (H1)-(H5) required in Section 2.3 of \cite{PWZZ}.
\end{proof}
\subsection{Verification of the conditions (\hyperref[C1]{C1})-(\hyperref[C4]{C4})}\label{section2}

{Let us} consider the Gelfand triplet $H^2\cap H_{0,\sigma}^1(\mathcal{O}_0)\hookrightarrow H_{0,\sigma}^1(\mathcal{O}_0)\hookrightarrow L^2_\sigma(\mathcal{O}_0)$. In this subsection, we will {build} a family of inner products on the space $H^1_{0,\sigma}(\mathcal{O}_0)$ and verify conditions (\hyperlink{C1}{C1})-(\hyperlink{C4}{C4}) needed in Section 2.2 of \cite{PWZZ}. These conditions are important in the next subsection as well, when we establish the equivalence between the equation \eqref{gns} and a SPDE on a fixed domain.

Due to \eqref{01270721}, we can define a family of inner products $\{(\cdot,\cdot)_{1,t}\}_{t\in[0,T]}$ on the space $H^1_{0,\sigma}(\mathcal{O}_0)$ given by
\begin{align}\label{checkc1}
	(v,w)_{1,t}\ \dot{=}\ \int_{\mathcal{O}_0} h^{kl}h_{ij}
	\big(\nabla_k v\big)_i (\nabla_l w)_j dy=(\nabla\bar{v},\nabla\bar{w})_{L^2(\mathcal{O}_t)}, \quad v, w\in H_{0,\sigma}^1(\mathcal{O}_0).
\end{align}
Now define the  elliptic operator ${L}^\#_h$ such that for any $\mathbb{R}^2$-valued function $v$ on $\mathcal{O}_0$,
\begin{align}\label{02021304}
	[{L}^\#_hv]_l\dot{=} h_{kl}(L_h v)_k=\frac{\partial r_m}{\partial y_l}\frac{\partial}{\partial y_n}\Big(\frac{\partial}{\partial y_j}\big(v_p\frac{\partial r_m}{\partial y_p}\big) h^{jn}\Big).
\end{align}
We stress that the elliptic operators ${L}^\#_h$ are time-dependent because both functions $h,r$ depend on the time variable.
Then for the Leray projection $P_0$ on $\mathcal{O}_0$, we can define the operator $\widehat{P_0L^\#_{h}}:H_{0,\sigma}^1(\mathcal{O}_0)\rightarrow H^{-1}_{\sigma}(\mathcal{O}_0)$, such that for any $v,w\in  H_{0,\sigma}^1(\mathcal{O}_0)$,
\begin{align}\label{02252057}
	_{H_\sigma^{-1}(\mathcal{O}_0)}\langle \widehat{P_0L^\#_h}v, w \rangle_{H_{0,\sigma}^{1}(\mathcal{O}_{0})}=-\int_{\mathcal{O}_0}h^{kl}h_{ij}(\nabla_k v)_i (\nabla_l w)_j dy=-(v,w)_{1,t}.
\end{align}
Note that when $v\in H^2\cap H_{0,\sigma}^1(\mathcal{O}_0)$, using the integration by parts formula, {we can see that $\widehat{P_0L^\#_{h}}v = P_0L^\#_{h}v$. Hereafter, we will no longer distinguish between $\widehat{P_0L^\#_{h}}$ and $P_0L^\#_{h}$}. In particular, it follows from \eqref{02252057} that $-P_0L_h^\#$ is the Riesz dual operator for the Hilbert space $H_{0,\sigma}^1(\mathcal{O}_0)$ equipped with the inner product $(\cdot,\cdot)_{1,t}$, which is an isomorphism. In particular, corresponding to the case $t=0$, the Stokes operator on $\mathcal{O}_0$, defined by $A_0=P_0\Delta:H_{0,\sigma}^1(\mathcal{O}_0)\rightarrow H_{\sigma}^{-1}(\mathcal{O}_0)$, satisfies for any $v,w\in H_{0,\sigma}^1(\mathcal{O}_0)$,
\begin{align*}
	_{H_\sigma^{-1}(\mathcal{O}_0)}\langle A_0v,w \rangle_{H_{0,\sigma}^{1}(\mathcal{O}_{0})}=-(\nabla v,\nabla w)_{L^2(\mathcal{O}_0)}.
\end{align*}
Namely $-A_0$: $H_{0,\sigma}^1(\mathcal{O}_0)\rightarrow H^{-1}_{\sigma}(\mathcal{O}_0)$ is the Riesz dual operator for the inner product $(\cdot,\cdot)_{1,0}$, which is an isomorphism. Therefore, we can define the family of time-dependent  operators
\begin{equation}\label{ltddy}
	\iota_t^*=A_0^{-1}P_0L_h^\# :H_{0,\sigma}^1(\mathcal{O}_0)\rightarrow H_{0,\sigma}^1(\mathcal{O}_0),
\end{equation}
which is also an isomorphism on $H_{0,\sigma}^1(\mathcal{O}_0)$ and satisfies
\begin{align*}
	(\iota_t^*v,w)_{1,0}=(v,w)_{1,t}.
\end{align*}

As required in Section 2.3 of \cite{PWZZ}, we need to verify the following conditions for $(\cdot,\cdot)_{1,t}$ and $\iota_t^*$:
\begin{itemize}
	\item[\hypertarget{C1}{{\bf (C1)}}] There exists a constant $c_1\geq 1$ such that for any $t\in[0,T]$, the norm ${\|\cdot\|_{1,t}}$ generated by the inner product $(\cdot,\cdot)_{1,t}$ satisfies,
	\begin{align*}
		{\frac{1}{c_1}\|x\|_{1,0}\leq \|x\|_{1,t}\leq c_1\|x\|_{1,0},\text{ for any $x\in H_{0,\sigma}^1(\mathcal{O}_0)$}. }
	\end{align*}
	\item[\hypertarget{C2}{{\bf (C2)}}] There exists a family of self-adjoint bounded operators $\{\Phi(t)\}_{t\in[0,T]}$ on ${H_{0,\sigma}^1(\mathcal{O}_0)}$ such that
	\begin{align*}
		\int_{0}^{T}\|\Phi(t)\|_{\mathcal{L}({H_{0,\sigma}^1(\mathcal{O}_0)})}dt<\infty\  ,
	\end{align*}
	and for any $t\in[0,T]$,
	\begin{align}\label{08201354}
		\|x\|_{1,t}^2-\|x\|_{1,0}^2= \int_{0}^{t}(x,\Phi(s)x)_{{1,0}}ds. 
	\end{align}
	\item[\hypertarget{C3}{{\bf (C3)}}]
	The operator $\iota_t^*$ is  a bounded operator on $H^2\cap H_{0,\sigma}^1(\mathcal{O}_0)$ as well, and there exists a constant $c_2>0$ such that
	\begin{align*}
		\|\iota_t^* x\|_{{H^2(\mathcal{O}_0)}}\leq c_2\|x\|_{{H^2(\mathcal{O}_0)}}\text{ for any $x\in {H^2\cap H_{0,\sigma}^1(\mathcal{O}_0)},  t\in[0,T]$}. 
	\end{align*}
	\item[\hypertarget{C4}{{\bf (C4)}}] $\iota_t^*$ is bijective from $H^2\cap H_{0,\sigma}^1(\mathcal{O}_0)$ to $H^2\cap H_{0,\sigma}^1(\mathcal{O}_0)$ and there exists a constant $c_3>0$ such that
	\begin{align*}
		\text{  $\|\iota_{-t}^*x\|_{H^2(\mathcal{O}_0)}\leq c_3\|x\|_{H^2(\mathcal{O}_0)}$ for any $x\in {H^2\cap H_{0,\sigma}^1(\mathcal{O}_0)}$, $t\in[0,T]$,} 
	\end{align*}
where $\iota_{-t}^*$ stands for the inverse of $\iota_{t}^*$.
\end{itemize}
\noindent
{\bf{Verification of (\hyperlink{C1}{C1}):}} Using the assumptions about the family of domains $\{\mathcal{O}_t\}_{t\in[0,T]}$, combining the Rayleigh-Faber-Krahn inequality and the relation \eqref{checkc1}, it is easy to verify that condition (\hyperlink{C1}{C1}) holds.
\vskip 0.3cm
\noindent{\bf{Verification of (\hyperlink{C2}{C2}):}} We define the time derivative operatior of $L_h^\#$ as $\Psi(t)$, that is,
\begin{equation}\label{PSIdefinition}
	\Psi(t)=\frac{d}{dt}L_h^\#(t):[H_{0}^1(\mathcal{O}_0)]^2\rightarrow [H^{-1}(\mathcal{O}_0)]^2.
\end{equation}
For an $\mathbb{R}^2$-valued function $v$ on $\mathcal{O}_0$,  $\Psi(t)v$ has the following explicit expression:
\begin{align*}
	\big(\Psi(t)v\big)_l =&\ \frac{\partial^2 r_m}{\partial y_l\partial t}\frac{\partial}{\partial y_n}\Big(\frac{\partial}{\partial y_j}\big(v_\alpha\frac{\partial r_m}{\partial y_\alpha}\big)h^{jn}\Big)+\frac{\partial r_m}{\partial y_l}\frac{\partial}{\partial y_n}\Big(\frac{\partial}{\partial y_j}\big(v_\alpha\frac{\partial r_m}{\partial y_\alpha}\big)\frac{\partial}{\partial t}h^{jn}\Big)\\&+\frac{\partial r_m}{\partial y_l}\frac{\partial}{\partial y_n}\Big(\frac{\partial}{\partial y_j}\big(v_\alpha\frac{\partial^2 r_m}{\partial y_\alpha\partial t}\big)h^{jn}\Big).
\end{align*}
{As in} (\ref{02252057}), we first define the operator
$P_0\Psi(t):H_{0,\sigma}^1(\mathcal{O}_0)\rightarrow H_{\sigma}^{-1}(\mathcal{O}_0)$, such that for any $v, w\in H_{0,\sigma}^1(\mathcal{O}_0)$,
\begin{align*}
	_{H^{-1}_\sigma(\mathcal{O}_0)}\langle P_0\Psi(t)v,w \rangle_{H^{1}_{0,\sigma}(\mathcal{O}_0)}=-\frac{d}{dt}\int_{\mathcal{O}_0} h^{kl}h_{ij}(\nabla_k v)_i(\nabla_l w)_j dy=-\frac{d}{dt}(v,w)_{1,t}.
\end{align*}
Then we define the operator $\Phi(t)=A_0^{-1}P_0\Psi(t):H_{0,\sigma}^1(\mathcal{O}_0)\rightarrow H_{0,\sigma}^1(\mathcal{O}_0)$, such that $\forall v,w\in H_{0,\sigma}^1(\mathcal{O}_0)$,
\begin{align*}
	&(\Phi(t)v,w)_{1,0}\\=&\int_{\mathcal{O}_0}\big(\frac{\partial v_i}{\partial y_k}\frac{\partial^2 r_\alpha}{\partial y_i\partial t}+v_m\frac{\partial^3 r_\alpha}{\partial t\partial y_k\partial y_m}\big)h^{kl}\big(\frac{\partial w_j}{\partial y_l}\frac{\partial r_\alpha}{\partial y_j}+w_n\frac{\partial^2 r_\alpha}{\partial y_l\partial y_n}\big)dy\\&+\int_{\mathcal{O}_0}\big(\frac{\partial v_i}{\partial y_k}\frac{\partial r_\alpha}{\partial y_i}+v_m\frac{\partial^2 r_\alpha}{\partial y_k\partial y_m}\big)h^{kl}\big(\frac{\partial w_j}{\partial y_l}\frac{\partial^2 r_\alpha}{\partial y_j\partial t}+w_n\frac{\partial^3 r_\alpha}{\partial y_l\partial y_n\partial t}\big)dy\\&+\int_{\mathcal{O}_0}\big(\frac{\partial v_i}{\partial y_k}\frac{\partial r_\alpha}{\partial y_i}+v_m\frac{\partial^2 r_\alpha}{\partial y_k\partial y_m}\big)\frac{\partial h^{kl}}{\partial t}\big(\frac{\partial w_j}{\partial y_l}\frac{\partial r_\alpha}{\partial y_j}+w_n\frac{\partial^2 r_\alpha}{\partial y_l\partial y_n}\big)dy\\=&\frac{d}{dt}(v,w)_{1,t}.
\end{align*}
This {also} shows that $\Phi(t)$ is a self-adjoint operator. Moreover, by the assumptions on the family of domains $\{\mathcal{O}_t\}_{t\in[0,T]}$, one can show that there exists a constant $C>0$ such that for every $t \in [0,T]$,
\begin{align*}
	\|\Phi(t)\|_{\mathcal{L}\big(H_{0,\sigma}^1(\mathcal{O}_0)\big)}\leq C.
\end{align*}
Consequently, we confirm the validity of (\hyperlink{C2}{C2}).
\vskip 0.3cm
\noindent{\bf{Verification of (\hyperlink{C3}{C3}) and (\hyperlink{C4}{C4}):}} For $g \in H^2\cap H_{0,\sigma}^1(\mathcal{O}_0)$, let $u=A_0^{-1}P_0L^\#_hg$. Then verifying condition (\hyperlink{C3}{C3}) reduces to obtaining an $L^2$ estimate for the solution of the following Stokes problem:
\begin{equation}
	\begin{cases}\label{01272301}
		\Delta u+\nabla p =P_0L^\#_hg, x\in \mathcal{O}_0,\\
		\mathrm{div}\ u=0, x\in \mathcal{O}_0,\\
		u=0, x\in \partial\mathcal{O}_0.
	\end{cases}
\end{equation}
In fact, by the well-posedness of the steady Stokes problem and the assumptions on the moving domains, we can find a constant $C>0$ such that
\begin{align}\label{01272255}
	\|\iota_t^* g\|_{{H^2(\mathcal{O}_0)}}=\|u\|_{H^2(\mathcal{O}_0)}\leq C_0\|P_0L_h^\#g\|_{L^2(\mathcal{O}_0)}\leq C\|g\|_{H^2(\mathcal{O}_0)}.
\end{align}
Thus, (\hyperlink{C3}{C3}) is verified. Now we consider condition (\hyperlink{C4}{C4}). Unlike (\hyperlink{C3}{C3}), if we replace $\Delta$ in \eqref{01272301} by $L_h^\#$, we no longer have an estimate similar to \eqref{01272255}. To the best of our knowledge, the existing literature on $L^2$ estimates for generalized Stokes problems typically assumes that the elliptic operator satisfies a uniform ellipticity condition, or takes a form that is significantly different from $L_h^\#$ (see e.g. \cite{BS,M} and references therein). In this case, we use the perturbation theory of operators to ensure that (\hyperlink{C4}{C4}) holds on a short time interval. We observe that for any $v\in {H^2\cap H_{0,\sigma}^1(\mathcal{O}_0)}$,
\begin{align*}
	[(L^\#_h-\Delta)v]_l=\ &\frac{\partial r_m}{\partial y_l}\frac{\partial}{\partial y_n}\big(\frac{\partial}{\partial y_j}(v_\alpha\frac{\partial r_m}{\partial y_\alpha})h^{jn}\big)-\delta_{jn}\frac{\partial}{\partial y_n}(\frac{\partial}{\partial y_j}v_l)\\=&\big(\frac{\partial r_m}{\partial y_l}-\delta_{ml}\big)\big[\frac{\partial}{\partial y_n}\big(\frac{\partial}{\partial y_j}(v_\alpha\frac{\partial r_m}{\partial y_\alpha})h^{jn}\big)\big]+\frac{\partial}{\partial y_n}\big[\frac{\partial}{\partial y_j}\big(v_\alpha\frac{\partial r_l}{\partial y_\alpha}(h^{jn}-\delta_{jn})\big)\big]\\&+\frac{\partial^2}{\partial y_j^2}\big(v_\alpha(\frac{\partial r_l}{\partial y_\alpha}-\delta_{\alpha l})\big)\\=&\mathrm{I+II+III}.
\end{align*}
From the regularity assumptions on $r$ and $\bar{r}$, we can infer that there exist a time $t_0>0$ and a constant $C_0 > 0$ such that for all $t\in [0, t_0]$,
$$\|\mathrm{I}\|_{L^2(\mathcal{O}_0)}+\|\mathrm{II}\|_{L^2(\mathcal{O}_0)}+\|\mathrm{III}\|_{L^2(\mathcal{O}_0)}\leq \frac{1}{2C_0}\|v\|_{H^2(\mathcal{O}_0)}. $$
It follows that
\begin{align}\label{01292226}
	\nonumber&\|(\iota_t^*-\iota_0^*)v\|_{{H^2(\mathcal{O}_0)}}\leq \|\big(A_0^{-1}P_0L_h^\#-A_0^{-1}P_0\Delta_0)v\|_{H^2(\mathcal{O}_0)} \\ &\leq\  C_0\|(P_0L_h^\#-P_0\Delta_0)v\|_{L^2_\sigma(\mathcal{O}_0)}\leq \frac{1}{2}\|v\|_{{H^2(\mathcal{O}_0)}}.
\end{align}
Since $\iota_0^*$ is the identity map, by perturbation theory, we know that $\iota_t^*$ is an isomorphism from $H^2\cap H_{0,\sigma}^1(\mathcal{O}_0)$ to $H^2\cap H_{0,\sigma}^1(\mathcal{O}_0)$, and that $\|\iota_{-t}^*v\|_{{H^2(\mathcal{O}_0)}}\leq 2\|v\|_{{H^2(\mathcal{O}_0)}}$. We complete the verification of (\hyperlink{C1}{C1})-(\hyperlink{C4}{C4}) on a small time interval $[0,t_0]$.

\begin{remark}\label{remark2}
	Due to the fact that $\iota_t^*$ is an isomorphism on $H^2\cap H_{0,\sigma}^1(\mathcal{O}_0)$ and $A_0$ is an isomorphism from $H^2\cap H_{0,\sigma}^1(\mathcal{O}_0)$ to $L_\sigma^2(\mathcal{O}_0)$, it follows that for any $t \in [0, t_0]$, the operator
	$$P_0L_h^\#:H^2\cap H_{0,\sigma}^1(\mathcal{O}_0)\rightarrow L^2_\sigma(\mathcal{O}_0)$$
	is also an isomorphism, and the operator norms of both itself and its inverse admit uniform upper bounds with respect to $t \in [0, t_0]$.
\end{remark}

\subsection{Equivalence of solutions}\label{section3}

Recall the modified Navier-Stokes equation, where $g_N(r)=1\wedge \frac{N}{r}$,
\vskip -0.5cm
\begin{numcases}{}\label{4.2.0}
	\nonumber\dot{u}(t,x)+g_N(\|\nabla u\|_{L^2(\mathcal{O}_t)})(u\cdot\nabla) u(t,x)+\nabla p=\Delta u(t,x)+\sigma\big(t,u(t)\big)\dot{W},\ (t,x)\in \mathcal{D}_T,\\
	\nonumber\mathrm{div}\ u(t,x)=0,\ (t,x)\in \mathcal{D}_T,\\
	\nonumber u(t,x)=0,\ \ \ \ \ \ (t,x)\in \partial\mathcal{D}_T,\\
	u(0,y)=u_0(y), \ \ \ \ \ y\in \mathcal{O}_0.
\end{numcases}

In this subsection, we will show that the solution of the globally modified Navier-Stokes equation \eqref{4.2.0} can be transformed into a solution of a SPDE with nonhomogeneous monotonicity in a variational setting and vice versa. Consider the Gelfand triplet $H^2\cap H_{0,\sigma}^1(\mathcal{O}_0)\hookrightarrow H_{0,\sigma}^1(\mathcal{O}_0)\hookrightarrow L^2_\sigma(\mathcal{O}_0)$. From Subsection \ref{section2}, we know that condition (\hyperlink{C4}{C4}) is only verified on the small time interval $[0,t_0]$. Therefore, we will focus on the equivalence of solutions on the interval $[0,t_0]$. Before we state the equivalence of the solutions, we need some notations. For a function  $v:\mathcal{O}_0\rightarrow \mathbb{R}^2$, recall the operator $N(v,v)$ defined in Section 1:
\begin{align}\label{4.2.1}
	 N(v,v)^i \dot{=}{v}_j({\nabla_j {v}})_i.
\end{align}
Let  $M$ be an operator defined by
$$\big(Mv\big)_i=\frac{\partial \bar{r}_k}{\partial t}(\nabla_kv)_i+\frac{\partial \bar{r}_i}{\partial x_k}\frac{\partial^2 r_k}{\partial t\partial y_j}v_j$$

Define a family of equivalent inner products $\{(\cdot,\cdot)_{0,t}\}_{t\in[0,T]}$ on $L^2_\sigma(\mathcal{O}_0)$ such that for any $v,w\in L^2_\sigma(\mathcal{O}_0)$,
	\begin{align}\label{njddy}
		(v,w)_{0,t}\dot{=}\int_{\mathcal{O}_0}h_{ij}(t,y)v_iw_j dy.
	\end{align}
	In particular, for the inverse Piola transformations of $v$ and $w$, which are denoted by $\bar{v}$ and $\bar{w}$ respectively, we have
	\begin{align}\label{03302301}
		(\bar{v},\bar{w})_{L^2(\mathcal{O}_t)}= (v,w)_{0,t}.
	\end{align}
Using the Leray projection $P_0$, we can define the operator $P_0h: L^2(\mathcal{O}_0)\rightarrow L^2_\sigma(\mathcal{O}_0)$ so that
\begin{align*}
		(v, P_0hv)_{L^2(\mathcal{O}_0)}=(v,hv)_{L^2(\mathcal{O}_0)}=(v,v)_{0,t}.
	\end{align*}
It is easy to see that there exists a constant $C>0$ independent of $t$ such that
	\begin{align}\label{01310039}
		\|P_0hv\|_{L^2(\mathcal{O}_0)}\in\big[\frac{1}{C}\|v\|_{L^2(\mathcal{O}_0)},C\|v\|_{L^2(\mathcal{O}_0)}\big].
	\end{align}
Moreover, $P_0h|_{L_\sigma^2(\mathcal{O}_0)}: L_\sigma^2(\mathcal{O}_0)\rightarrow L_\sigma^2(\mathcal{O}_0)$ is a strictly positive definite isomorphism.
	So $\left(P_0h|_{L_\sigma^2(\mathcal{O}_0)}\right)^{-1}$ is also a strictly positive definite operator, and $\|\left(P_0h|_{L_\sigma^2(\mathcal{O}_0)}\right)^{-1}\|_{\mathcal{L}(L^2_\sigma(\mathcal{O}_0))}$ is uniformly bounded with respect to $t$.
The time derivative operator of $P_0h|_{L_\sigma^2(\mathcal{O}_0)}$, denoted by $\Phi_0(t): L^2_\sigma(\mathcal{O}_0)\rightarrow L^2_\sigma(\mathcal{O}_0)$,  is given by:
	\begin{align}\label{01292005}
		v\in L^2_\sigma(\mathcal{O}_0) \rightarrow P_0\frac{\partial h_{\cdot j}}{\partial t}v_j \in L^2_\sigma(\mathcal{O}_0).
	\end{align}
To simplify the notation, in the sequel we will write $(P_0h)^{-1}$ instead of $\left(P_0h|_{L_\sigma^2(\mathcal{O}_0)}\right)^{-1}$. We stress that with this notation, $(P_0h)^{-1}(P_0h)\not =I$.
	\begin{proposition}\label{5.1}
		The globally modified Navier-Stokes equation \eqref{4.2.0} on $\mathcal{D}_{t_0}$ is equivalent to the following SPDE in the Gelfand triplet $H^2\cap H_{0,\sigma}^1(\mathcal{O}_0)\hookrightarrow H_{0,\sigma}^1(\mathcal{O}_0)\hookrightarrow L^2_\sigma(\mathcal{O}_0)$:
			\begin{numcases}{}
				\label{01301253} dv(t)=(P_0h)^{-1}(P_0L_h^\#v(t))dt-g_N(\|v\|_{1,t})(P_0h)^{-1}(P_0h)N(v,v)(t)dt\\\ \ \ \ \ \ \ \ \ \ -(P_0h)^{-1}(P_0h)Mv(t)dt+\widetilde{\sigma}(t,v(t))dW_t,\ t\in[0,t_0]\nonumber\\
				v(0)=v_0\in H^1_{0,\sigma}(\mathcal{O}_0).\nonumber
			\end{numcases}
		where $W=\{ \beta^k \}_{k \geq 1}$ is the same $l^2$ cylindrical Wiener process as in equation \eqref{ns}.\par
	\end{proposition}

\begin{proof}
	
Let $u^N$ be a solution to \eqref{4.2.0}. Then for any $\varphi\in C_{c,\sigma}^2({\mathcal{D}_{t_0}})$, $(u^N,\varphi)$ satisfies \eqref{ws2'}. Denote $(\widetilde{u^N},\widetilde{\varphi})$ by $(v,\psi)$. We will show that $v$ is a solution to the SPDE (\ref{01301253}) {in} the Gelfand triplet $H^2\cap H_{0,\sigma}^1(\mathcal{O}_0)\hookrightarrow H_{0,\sigma}^1(\mathcal{O}_0)\hookrightarrow L^2_\sigma(\mathcal{O}_0)$.
	
	By \eqref{01290152} and \eqref{03302301}, we have
	\begin{align}\label{01290341}
		\int_{0}^{t}\int_{\mathcal{O}_s}u^N(s)\cdot\varphi'(s)dxds
		&=\int_{0}^{t}\big(v(s),\widetilde{\varphi'}(s)\big)_{0,s}ds\nonumber\\
		&=\int_{0}^{t}\big(v(s),\frac{\partial \psi}{\partial s}(s)\big)_{0,s} ds+\int_{0}^{t}\big(v(s),M\psi(s)\big)_{0,s} ds,
	\end{align}
	where $\big(M\psi\big)_i=\frac{\partial \bar{r}_k}{\partial t}(\nabla_k\psi)_i+\frac{\partial \bar{r}_i}{\partial x_k}\frac{\partial^2 r_k}{\partial t\partial y_j}\psi_j$.
	
	For the stochastic term, by \eqref{03302301} , we have
	\begin{align*}
		\int_{0}^{t}\int_{\mathcal{O}_s}\sigma_k(s,u^N(s))\varphi(s,x)dxd\beta^k_s=\int_{0}^{t}\big(\widetilde{\sigma}_k(s,v(s)),\psi(s)\big)_{0,s}d\beta^k_s,
	\end{align*}
	where {$\widetilde{\sigma}_k(s,v)\ \dot{=}\ \widetilde{\sigma_k(s,\bar{v})}$}. 
	Similarly, it follows from \eqref{01270721} and \eqref{01291431} that
	\begin{align}
		\label{01291451}&\int_{0}^{t_0}\int_{{\mathcal{O}_s}}\Delta u^N(s,x)\varphi(s,x)dxds=\int_{0}^{t_0}\big(L_hv,\psi(s)\big)_{0,s}ds,\\
		&\nonumber\int_0^{t_0} \int_{{\mathcal{O}_s}} g_N(\|u^N\|_{H^1(\mathcal{O}_s)}) (u^N\cdot\nabla\phi)(s,x)\cdot u^N(s,x) dxds\\=\ &\int_{0}^{t_0}\int_{{\mathcal{O}_0}}g_N(\|v\|_{1,s})v_j(\nabla_j \psi)_k v_l h_{kl}dyds=\int_{0}^{t_0}g_N(\|v\|_{1,s})\big(N(v,\psi), v\big)_{0,s}ds.	\label{01291452}
	\end{align}
	In particular, for {operators} $\Phi_0$ defined in \eqref{01292005} and $M$ {appeared} in \eqref{01290341}, applying the integration by parts formula, we find the following relation: for any $v\in H_{0,\sigma}^1\big(\mathcal{O}_0\big)$,
	\begin{align*}
		(v, Mv)_{0,s}=\frac{1}{2}\big(\Phi_0(s)v,v\big),
	\end{align*}
	which implies that for any $v,w\in H_{0,\sigma}^1\big(\mathcal{O}_0\big)$,
	\begin{align}\label{01312053}
		(v, Mw)_{0,s}=\big(\Phi_0(s)v,w\big)_{L^2}-(w, Mv)_{0,s}.
	\end{align}
	Combining \eqref{01312053} with \eqref{01290341}-\eqref{01291452}, we get
	\begin{align}\label{01300248}
		\nonumber&\big(v(t),\psi(t)\big)_{0,t}-\big(v(0),\psi(0)\big)_{0,0}\\\nonumber=&\ \int_{0}^{t}\big(v(s),\frac{\partial\psi}{\partial s}(s)\big)_{0,s} ds-\int_{0}^{t}\big(Mv(s),\psi(s)\big)_{0,s} ds+\int_{0}^{t}\big(\Phi_0v(s),\psi(s)\big)_{L^2}ds+\int_{0}^{t}\big(L_hv(s),\psi(s)\big)_{0,s}ds\\&\ +\int_{0}^{t}g_N(\|v\|_{1,s})\big(N(v,\psi)(s),v(s)\big)_{0,s}ds+\sum_{k=1}^{\infty}\int_{0}^{t}\big(\widetilde{\sigma}_k(s,v(s)),\psi(s)\big)_{0,s}d\beta^k_s.
	\end{align}
	
	Let $\psi\in C^2\big([0,t_0],C_{c,\sigma}^2({\mathcal{O}_0})\big)$ be arbitrary. By Remark \ref{remark2}, we have that for $t \in [0,t_0]$, the operator $P_0L^\#_h:H^2\cap H_{0,\sigma}^1(\mathcal{O}_0)\rightarrow L^2_\sigma(\mathcal{O}_0)$ is an isomorphism. Based on the assumptions on the moving domains, we deduce that
	$$P_0L_h^\#\in C^1\big([0,t_0],\mathcal{L}(H^2\cap H^1_{0,\sigma}(\mathcal{O}_0), L_\sigma^2(\mathcal{O}_0))\big),\frac{d}{dt}P_0L_h^\#=P_0\Psi(t),$$
	{where $\Psi$ is defined as in \eqref{PSIdefinition}.}
	Moreover, there exists $\phi(\cdot)\in C^1\big([0,t_0],H^2\cap H_{0,\sigma}^1(\mathcal{O}_0)\big)$ such that for any $t \in [0,t_0]$,
	$$-P_0L_h^\#\phi(t)=P_0h\psi(t).$$
	Then
	\begin{align}\label{01302124}
		&\big(v(t),\psi(t)\big)_{0,t}=-\int_{{\mathcal{O}_0}}v_i (P_0L_h^\#\phi(t))_i dy=-\int_{{\mathcal{O}_0}} v_i(L_h^\#\phi(t))_i dy=\big(v(t),\phi(t)\big)_{1,t}.
	\end{align}
	In particular, corresponding to the case $t=0$,
	\begin{align*}
		\big(v(0),\psi(0)\big)_{0,0}=\big(v(0),\phi(0)\big)_{1,0}.
	\end{align*}
	Since
	\begin{align}\label{01301049}
		\psi(t)=-(P_0h)^{-1}(P_0L_h^\#\phi)(t),
	\end{align}
	it follows that
	\begin{align}\label{01302123}
		\frac{\partial \psi}{\partial s}(s)&=(P_0h)^{-1}P_0\frac{\partial h}{\partial s}(P_0h)^{-1}\big(P_0L_h^\#\phi(s)\big)-(P_0h)^{-1}\big(P_0\Psi(s)\phi(s)\big)\\&\ \ \ -(P_0h)^{-1}\big({P}_0L_h^\#\frac{\partial \phi}{\partial s}\big)(s)\nonumber.
	\end{align}
	Therefore, we have
	\begin{align}
		\nonumber&\big(v(s),\frac{\partial \psi}{\partial s}(s)\big)_{0,s}=\big((P_0h)v(s),\frac{\partial \psi}{\partial s}(s)\big)_{0,0}\\=&\ \nonumber\big(v(s),P_0\frac{\partial h}{\partial s}(P_0h)^{-1}(P_0L_h^\#\phi(s))\big)_{0,0} -\big(v(s),P_0\Psi(s)\phi(s)\big)_{0,0}+\big(v(s),\frac{\partial \phi}{\partial s}(s)\big)_{1,s}\\=&\ -\big(v(s),\Phi_0(s)\psi(s)\big)_{0,0}+\big(v(s),\Phi(s)\phi(s)\big)_{1,0}+\big(v(s),\frac{\partial \phi}{\partial s}(s)\big)_{1,s}.
	\end{align}
	As for the remaining terms, we first observe that
	\begin{align}\label{03302238}
		\nonumber\big(L_hv(s),\psi(s)\big)_{0,s}&=-\big(P_0h\big(L_hv(s)\big),(P_0h)^{-1}P_0{L}_h^\#\phi(s)\big)_{0,0}
		\\&=-\big((P_0h)^{-1}P_0h(L_hv(s)),P_0L_h^\#\phi(s)\big)_{0,0}\nonumber\\
		&=-\big((P_0h)^{-1}P_0L_h^\#v(s),P_0L_h^\#\phi(s)\big)_{0,0}.
	\end{align}
	Then using a similar argument, we infer that
	\begin{align}
		\label{01302127-1}&\big(Mv(s),\psi(s)\big)_{0,s}=-\big((P_0h)^{-1}P_0hMv(s),P_0L_h^\#\phi(s)\big)_{0,0},\\& \label{01302127-2}\big(\widetilde{\sigma}_k(s,v(s)),\psi(s)\big)_{0,s}=\big(\widetilde{\sigma}_k(s,v(s)),\phi(s)\big)_{1,s},\\&
		\label{01302127-3}\big(N(v,v)(s),\psi(s)\big)_{0,s}=-\big((P_0h)^{-1}P_0h N(v,v)(s), P_0L_h^\# \phi(s)\big)_{0,0}.
	\end{align}
	Combining \eqref{01302124}-\eqref{01302127-3}, we obtain
	\begin{align}\label{01301150}
		&\big(v(t),\phi(t)\big)_{1,t}-\big(v_0,\phi(0)\big)_{1,0}=\int_{0}^{t}\big(v(s),\frac{\partial\phi}{\partial s}\big)_{1,s}ds+\int_{0}^{t}\big(v(s),\Phi(s)\phi(s)\big)_{1,0}ds\\&\nonumber-\int_{0}^{t}\big((P_0h)^{-1}P_0L_h^\#v,P_0L_h^\#\phi(s)\big)_{0,0} ds+\int_{0}^{t}\big((P_0h)^{-1}P_0hMv,P_0L_h^\#\phi(s)\big)_{0,0} ds\\&+\int_{0}^{t}g_N(\|v\|_{1,s})\big((P_0h)^{-1}P_0hN(v,v),P_0L_h^\#\phi(s)\big)_{0,0} ds+\sum_{k=1}^{\infty}\int_{0}^{t}\big(\widetilde{\sigma}_k(s,v(s)),\phi(s)\big)_{1,s}d\beta_s^k\nonumber,
	\end{align}
	for any $\phi$ obtained through \eqref{01301049}. On the other hand, for any $\phi\in C^1\big([0,t_0],H^2\cap H^1_{0,\sigma}(\mathcal{O}_0)\big)$, \begin{align*}
		\psi\dot{=}-(P_0h)^{-1}(P_0L_h^\#\phi)\in C^1\big([0,t_0],L_\sigma^2(\mathcal{O}_0)\big).
	\end{align*}
	Take a sequence $\psi_n\in C^2\big([0,t_0],C_{c,\sigma}^2(\mathcal{O}_0)\big)$ such that $\psi_n$ converges to $\psi$ in $H^1\big([0,t_0],L_\sigma^2(\mathcal{O}_0)\big)$.
	Then we have
	$$\ {\phi}_n\dot{=}-(P_0L_h^\#)^{-1}P_0h\psi_n \rightarrow \phi \ \text{in} \ H^1\big([0,t_0],H^2\cap H^1_{0,\sigma}(\mathcal{O}_0)\big).   $$
	Obviously, $\phi_n$ is obtained through \eqref{01301049}. By taking $\phi = \phi_n$ in \eqref{01301150} and letting $n\rightarrow\infty$, we can deduce that \eqref{01301150} holds for any $\phi\in C^1([0,t_0],H^2\cap H_{0,\sigma}^1(\mathcal{O}_0))$. Therefore, $v$ satisfies the following equation {in} the Gelfand triplet $H^2\cap H_{0,\sigma}^1(\mathcal{O}_0)\hookrightarrow H_{0,\sigma}^1(\mathcal{O}_0)\hookrightarrow L^2_\sigma(\mathcal{O}_0)$:
	\begin{equation}
		\begin{cases}
			d \iota_t^*v(t)=\tilde{A}_1(t,v(t))dt+\tilde{A}_2(t,v(t))dt+\tilde{A}_3(t,v(t))dt+\iota_t^*\widetilde{\sigma}(t,v(t))dW_t+\Phi(t)v(t)dt,\ t\in[0,t_0],\\
			v(0)=v_0\in H_{0,\sigma}^1(\mathcal{O}_0),
		\end{cases}
	\end{equation}
	where the operator $\iota_t^*$ is defined in \eqref{ltddy}, $W=\{ \beta^k \}_{k \geq 1}$ represents an $l^2$-cylindrical Wiener process, $\tilde{A}_1(\cdot, \cdot), \tilde{A}_2(\cdot, \cdot), \tilde{A}_3(\cdot, \cdot): [0, t_0] \times H^2\cap H_{0,\sigma}^1(\mathcal{O}_0) \rightarrow L^2_\sigma(\mathcal{O}_0)$ are {respectively} given by
	\begin{align}
		\label{01310005-1}\langle\tilde{A}_1(t,v),w\rangle&=-\int_{\mathcal{O}_0}(P_0h)^{-1}(P_0L^\#_hv)(y)\cdot (P_0L^\#_hw)(y)dy,\\ \label{01310005-2}\langle\tilde{A}_2(t,v),w\rangle&=\int_{\mathcal{O}_0}(P_0h)^{-1}(P_0h)g_N(\|v\|_{1,t})N(v,v)(y)\cdot (P_0L^\#_hw)(y)dy,\\
		\label{01310005-3}\langle\tilde{A}_3(t,v),w\rangle&=\int_{\mathcal{O}_0}(P_0h)^{-1}(P_0h)Mv(y)\cdot (P_0L^\#_hw)(y)dy,
	\end{align}
	for $v,w\in H^2\cap H_{0,\sigma}^1(\mathcal{O}_0)$.
	{Let us} denote $\iota_{-\cdot}^*\tilde{A}_i(\cdot, v)$ by $A_i(\cdot, v)$ for $v \in H^2\cap H_{0,\sigma}^1(\mathcal{O}_0) $, $i=1,2,3$. From the definition of $\iota_{-\cdot}^*$, we see that
\begin{align*}
&A_1(t,v(t))=(P_0h)^{-1}(P_0L_h^\#v(t)), \quad A_2(t,v(t))=-g_N(\|v\|_{1,t})(P_0h)^{-1}(P_0h)N(v,v)(t),\\&A_3(t,v(t))=-(P_0h)^{-1}(P_0h)Mv(t)
\end{align*}
We conclude that $v$ is the solution of the SPDE \eqref{01301253} in the Gelfand triplet $H^2\cap H_{0,\sigma}^1(\mathcal{O}_0)\hookrightarrow H_{0,\sigma}^1(\mathcal{O}_0)\hookrightarrow L^2_\sigma(\mathcal{O}_0)$.
	Thus, we have shown that any solution $u^N$ to equation \eqref{gns} can be transformed into a solution $v$ of the SPDE \eqref{01301253} by $v\dot{=}\widetilde{u^N}$.
	
	Conversely, if $v$ is a variational solution to equation \eqref{01301253} and let $u^N \dot{=}\bar{v}$. By Theorem 2.5 in \cite{PWZZ}, we obtain that for any $\phi\in L^2\big([0,t_0],H^2\cap H^1_{0,\sigma}(\mathcal{O}_0)\big)\cap C^1\big([0,t_0],H^1_{0,\sigma}(\mathcal{O}_0)\big)$,
	\begin{align*}
		&\big(v(t),\phi(t)\big)_{1,t}-\big(v_0,\phi(0)\big)_{1,0}\\=&\int_{0}^{t}\langle \tilde{A}_1(s,v(s)),\phi(s)\rangle ds+\int_{0}^{t}\langle \tilde{A}_2(s,v(s)),\phi(s)\rangle ds+\int_{0}^{t}\langle \tilde{A}_3(s,v(s)),\phi(s)\rangle ds\\&\ +\int_{0}^{t}\big(v(s),\phi'(s)\big)_{1,s}ds+\int_{0}^{t}\big(\widetilde{\sigma}(s,v(s))dW_s,\phi(s)\big)_{1,s}+\int_{0}^{t}\big(\Phi(s)v(s),\phi(s)\big)_{1,0}ds.
	\end{align*}
	In particular, take any $\varphi\in C^2_{c,\sigma}(\mathcal{D}_{t_0})$ and denote $\widetilde{\varphi}\in C_{c,\sigma}^2([0,t_0]\times\mathcal{O}_0)$ by $\psi$. Let $\phi=-(P_0L^\#_h)^{-1}(P_0h\psi)$. It is easy to see that
	$\phi\in C\big([0,t_0],H^2\cap H^1_{0,\sigma}(\mathcal{O}_0)\big)\cap C^1\big([0,t_0],H^1_{0,\sigma}(\mathcal{O}_0)\big)$.
	Combining \eqref{01302124}-\eqref{01302127-3}, we infer that $(v, \psi)$ satisfies \eqref{01300248}.
	Next, combining \eqref{03302301} with \eqref{01290341}-\eqref{01312053}, we can see that \eqref{ws2'} holds for $\varphi\in C^2_{c,\sigma}(\mathcal{D}_{t_0})$. The other conditions (i) and (ii) in Definition \ref{ws3} follow {easily}. Thus, the stochastic process $u^N$ on moving domains, obtained by $u^N \dot{=}\bar{v}$, where $v$ is the solution of \eqref{01301253}, is a solution of \eqref{gns} on $[0,t_0]$.
\end{proof}

\subsection{Well-posedness of the modified stochastic Navier-Stokes equation}\label{section4}

In this subsection, we will complete the proof of the well-posedness of the modified stochastic Navier-Stokes equation. Since we have established the correspondence between the solutions to equations \eqref{gns} and \eqref{01301253} on $[0,t_0]$, to obtain the well-posedness of the globally modified Navier-Stokes equation \eqref{gns}, it suffices to establish the well-posedness of the SPDE \eqref{01301253}. {In order to apply Theorem 2.8 in \cite{PWZZ} to get the well-posedness of the SPDE \eqref{01301253}, it remains to verify the conditions (H1)-(H5) below for $A_i$, $i=1,2,3$ and $\widetilde{\sigma}$ defined in \eqref{01301253}. As in the previous subsection, we consider the equation on the time interval $[0,t_0]$.  Denote $\tilde{A}(t, \cdot)\dot{=} \sum_{i=1}^3\iota_{t}^*A_i(t, \cdot) =\sum_{i=1}^3\tilde{A}_i(t, \cdot)$. we will verify that there exists $\tilde{f}\in L^1([0,t_0],\mathbb{R}_+)$ such that the following  holds for $\tilde{A}$ and $\widetilde{\sigma}$.
	\begin{itemize}
		\item[\hypertarget{H1}{{\bf (H1)}}] For a.e. $t \in [0,t_0]$, the map $\lambda \in \mathbb{R} \rightarrow \langle \tilde{A}(t, u+\lambda v),  x \rangle \in \mathbb{R}$ is continuous, for any $u, v, x \in H^2\cap H^1_{0,\sigma}(\mathcal{O}_0)$.
		\item[\hypertarget{H2}{{\bf (H2)}}] There exists a locally bounded measurable function $\rho : H^2\cap H^1_{0,\sigma}(\mathcal{O}_0) \rightarrow \mathbb{R}$, nonnegative constants $\gamma$ and $C$ such that for a.e. $t \in [0, t_0]$, the following inequalities hold for any $u, v \in H^2\cap H^1_{0,\sigma}(\mathcal{O}_0)$,
		\begin{eqnarray*}
			&     & 2\langle \tilde{A}(t, u)- \tilde{A}(t, v), u-v \rangle + \sum_{k=1}^{\infty}\|\widetilde{\sigma}_k(t,u)-\widetilde{\sigma}_k(t,v)\|^2_{1,t} \\
			&\leq &  [\tilde{f}(t)+ \rho(v)]\| u -v \|_{1,t}^2,\\
			&     &  |\rho(u)| \leq C(1+ \| u\|^2_{H^2(\mathcal{O}_0)})(1+ \|u \|_{1,0}^\gamma).
		\end{eqnarray*}
		\item[\hypertarget{H3}{{\bf (H3)}}] There exists a constant $c > 0$ such that for a.e. $t \in [0, t_0]$, the following inequality holds for any $u \in H^2\cap H^1_{0,\sigma}(\mathcal{O}_0)$,
		\[ 2 \langle  \tilde{A}(t, u), u \rangle + \sum_{k=1}^{\infty}\|\widetilde{\sigma}_k(t,u)\|_{1,t}^2
		\leq \tilde{f}(t)(1 + \|u \|_{1,t}^2) - c\| u \|^2_{H^2(\mathcal{O}_0)}. \]
		\item[\hypertarget{H4}{{\bf (H4)}}] There exist nonnegative constants $\beta$ and $C$ such that for a.e. $t \in [0, t_0]$, we have for any $u \in H^2\cap H^1_{0,\sigma}(\mathcal{O}_0)$,
		\[ \| \tilde{A}(t, u)\|^2_{L^2(\mathcal{O}_0)} \leq (\tilde{f}(t) + C\|u\|^2_{H^2(\mathcal{O}_0)})(1+ \|u\|_{1,0}^\beta).     \]
		\item[\hypertarget{H5}{{\bf (H5)}}] For a.e. $t \in [0, t_0]$, we have for any $u \in H^2\cap H^1_{0,\sigma}(\mathcal{O}_0)$,
		\[  \sum_{k=1}^{\infty}\|\widetilde{\sigma}_k(t,u)\|_{1,t}^2 \leq \tilde{f}(t)(1 + \|u \|_{1,t}^2).  \]
	\end{itemize}
}
We first note that \eqref{02062230-1} and \eqref{02062230-2} imply that for any $u, v \in H^2\cap H^1_{0,\sigma}(\mathcal{O}_0)$ and a.e. $t \in [0,T]$,
\begin{align}
	&\label{02062230-1'}\sum_{k=1}^{\infty}\|\widetilde{\sigma}_k(t,u)-\widetilde{\sigma}_k(t,v)\|^2_{1,t}\leq f(t)\|u-v\|^2_{1,t},\\\label{02062230-2'}&\sum_{k=1}^{\infty}\|\widetilde{\sigma}_k(t,u)\|_{1,t}^2\leq f(t)\big(1+\|u\|^2_{1,t}\big).
\end{align}
{Thus we only need to verify (\hyperlink{H1}{H1})-(\hyperlink{H4}{H4}) for $\tilde{A}(t, \cdot)$.}\\ 
{\bf{Verification of (\hyperlink{H1}{H1}):}} For operators $\tilde{A}_1$ and $\tilde{A}_3$, the condition (H1) follows directly from their linearity. As for $\tilde{A}_2$, since for any $v, w \in H^2\cap H^1_{0,\sigma}(\mathcal{O}_0)$,
\begin{align*}
	\langle\tilde{A}_2(t,v),w\rangle&=\int_{\mathcal{O}_0}(P_0h)^{-1}(P_0h)g_N(\|v\|_{1,t})N(v,v)(y)\cdot (P_0L^\#_hw)(y)dy,
\end{align*}
by the bilinearity of $N(v,v)$ and the continuity of $g_N(\|v\|_{1,t})$ with respect to $v$ in $H^2\cap H^1_{0,\sigma}(\mathcal{O}_0)$, the condition (H1) is also satisfied.\\
{\bf{Verification of (\hyperlink{H2}{H2}):}}
We consider the operators $\tilde{A}_i$, $i=1,2,3$ {separately}. For any $v,w\in H^2\cap H^1_{0,\sigma}(\mathcal{O}_0)$,
\begin{align*}
	\langle\tilde{A}_1(t,v)-\tilde{A}_1(t,w),v-w\rangle=-\int_{\mathcal{O}_0}(P_0h)^{-1}\big(P_0L^\#_h(v-w)\big)(y)\cdot \big(P_0L^\#_h(v-w)\big)(y)dy,
\end{align*}
Since $(P_0h)^{-1}$ is a strictly positive definite operator on $L^2_\sigma(\mathcal{O}_0)$, and by \eqref{01310039}, there exists a constant ${\tilde{c}}>0$ independent of $t$ such that
\begin{align*}
	\langle\tilde{A}_1(t,v)-\tilde{A}_1(t,w),v-w\rangle\leq -{\tilde c}\|P_0L_h^\#(v-w)\|_{L^2(\mathcal{O}_0)}^2.
\end{align*}
Moreover, due to Remark \ref{remark2}, we have that
\begin{align}\label{I1I1}
	\langle\tilde{A}_1(t,v)-\tilde{A}_1(t,w),v-w\rangle\leq -{c}\|v-w\|_{H^2(\mathcal{O}_0)}^2.
\end{align}
For $\tilde{A}_3$, {we have}
\begin{align}\label{I3I3}
	&\langle\tilde{A}_3(t,v)-\tilde{A}_3(t,w),v-w\rangle\\ \nonumber=&\int_{\mathcal{O}_0}(P_0h)^{-1}(P_0h)M(v-w)(y)\cdot (P_0L^\#_h(v-w))(y)dy\\\leq&\ \nonumber C\|v-w\|_{1,t}\|v-w\|_{H^2(\mathcal{O}_0)}\\\leq&\ \frac{c}{8}\|v-w\|^2_{H^2(\mathcal{O}_0)}+C\|v-w\|^2_{1,t}. \nonumber
\end{align}
As for $\tilde{A}_2$,
\begin{align*}
	&\langle\tilde{A}_2(t,v)-\tilde{A}_2(t,w),v-w\rangle\\=&\int_{\mathcal{O}_0}(P_0h)^{-1}(P_0h)\big(g_N(\|v\|_{1,t})N(v,v)(y)-g_N(\|w\|_{1,t})N(w,w)(y)\big)\cdot\big(P_0L^\#_h(v-w)\big)(y)dy\\=&\int_{\mathcal{O}_0}(P_0h)^{-1}(P_0h)g_N(\|v\|_{1,t})N(v-w,v)(y)\cdot\big(P_0L^\#_h(v-w)\big)(y)dy\\\ &+\int_{\mathcal{O}_0}(P_0h)^{-1}(P_0h)\big(g_N(\|v\|_{1,t})-g_N(\|w\|_{1,t})\big)N(w,v)(y)\cdot\big(P_0L^\#_h(v-w)\big)(y)dy\\\ &+\int_{\mathcal{O}_0}(P_0h)^{-1}(P_0h)g_N(\|w\|_{1,t})N(w,v-w)(y)\cdot\big(P_0L^\#_h(v-w)\big)(y)dy\\=&\mathrm{I+II+III}.
\end{align*}
From Lemma 2.1 in \cite{M}, {the following is valid},
\begin{itemize} {
	}    \item[(i)]: $0\leq \|v\|_{1,t}g_N(\|v\|_{1,t})\leq N$,
	\item[(ii)]: $|g_N(\|v\|_{1,t})-g_N(\|w\|_{1,t})|\leq \frac{1}{N}g_N(\|v\|_{1,t})g_N(\|w\|_{1,t})\|v-w\|_{1,t}$.
\end{itemize}
{Consequently}, we get
\begin{align}
	\label{I3I31}&\mathrm{I}\leq C\frac{N}{\|v\|_{1,t}}\|v-w\|_{1,t}^{\frac{1}{2}}\|v\|_{1,t}\|v-w\|_{H^2(\mathcal{O}_0)}^{\frac{3}{2}}\leq \frac{c}{8}\|v-w\|_{H^2(\mathcal{O}_0)}^2+CN^4\|v-w\|^2_{1,t}.\\\label{I3I32}&
	\mathrm{II}\leq\frac{C}{N}g_N(\|v\|_{1,t})g_N(\|w\|_{1,t})\|v-w\|_{1,t}\|v\|_{1,t}\|w\|_{1,t}^\frac{1}{2}\|w\|_{H^2(\mathcal{O}_0)}^\frac{1}{2}\|v-w\|_{H^2(\mathcal{O}_0)}\\& \ \ \leq \frac{c}{8}\|v-w\|_{H^2(\mathcal{O}_0)}^2+C\big(1+\|w\|_{H^2(\mathcal{O}_0)}\|w\|_{1,t}\big)\|v-w\|_{1,t}^2. \nonumber
	\\  \label{I3I33}&
	\mathrm{III}\leq Cg_N(\|w\|_{1,t})\|v-w\|_{1,t}\|w\|_{1,t}^\frac{1}{2}\|w\|_{H^2(\mathcal{O}_0)}^\frac{1}{2}\|v-w\|_{H^2(\mathcal{O}_0)}\\& \ \ \ \ \leq \frac{c}{8}\|v-w\|_{H^2(\mathcal{O}_0)}^2+C\big(1+\|w\|_{H^2(\mathcal{O}_0)}\|w\|_{1,t}\big)\|v-w\|_{1,t}^2. \nonumber
\end{align}
Combining with \eqref{I1I1}-\eqref{I3I33}, we finally obtain 
\begin{align}\label{I4I4}
	\langle \tilde{A}(t,v)-\tilde{A}(t,w), v-w\rangle\leq -\frac{c}{2}\|v-w\|^2_{H^2(\mathcal{O}_0)}+C_N\big(1+\|w\|_{H^2(\mathcal{O}_0)}\|w\|_{1,t}\big)\|v-w\|_{1,t}^2.
\end{align}\par
\noindent{\bf{Verification of (\hyperlink{H3}{H3}):}} Taking $w=0$ in \eqref{I4I4}, we have
\begin{align*}
	\langle \tilde{A}(t,v), v\rangle\leq-\frac{c}{2}\|v\|^2_{H^2(\mathcal{O}_0)}+C_N\|v\|_{1,t}^2.
\end{align*}\par
\noindent{\bf{Verification of (\hyperlink{H4}{H4}):}} By the definitions \eqref{01310005-1}-\eqref{01310005-3}, it easily follows that
$$\|\tilde{A}(t,v)\|_{L^2(\mathcal{O}_0)}^2\leq C(1+\|v\|^2_{H^2(\mathcal{O}_0)})(1+\|v\|^2_{1,t}).$$

The well-posedness of the globally modified Navier-Stokes equation \eqref{gns} on $[0,t_0]$ is thus proved.

\section{Well-posedness of the stochastic Navier-Stokes equation on  $[0,t_0]$}\label{section5}

In this section, we will establish the well-posedness of the two dimensional Navier-Stokes equation \eqref{ns} on the moving domain $\mathcal{D}_{t_0}$. We first prove the uniqueness of equation \eqref{ns}. To this end, we introduce the following lemma, which is a particular case of Theorem 2.5 in \cite{PWZZ}.
\begin{lemma}\label{ito2}
	Assume that $H$ is a separable Hilbert space with a family of equivalent inner products $\{(\cdot,\cdot)_{t}\}_{t\in[0,T]}$ satisfying conditions (\hyperlink{C1}{C1}) and (\hyperlink{C2}{C2}). Let  $X_0\in L^2(\Omega,\mathcal{F}_0,\mathbb{P},H)$, $Y\in L^1([0,T],H), \mathbb{P}-a.s.$  and  $Z\in L^2([0,T],L_2(U,H)), \mathbb{P}-a.s.$ both progressively measurable. If $X$  is an $H$-valued continuous process of the form,
	\begin{align}\label{01311944}
		X_t=X_0+\int_{0}^{t}Y_s ds+\int_{0}^{t}Z_sdW_s, \  \forall t\in[0,T].
	\end{align}
	Then $\mathbb{P}-a.s.$, the following equality holds
	\begin{align*}
		\nonumber|X_t|_t^2= \  &{|X_0|_0^2}+\int_{0}^{t}\big(2(Y_s, X_s)_s+\|Z_s\|^2_{L_2(U,H^s)}\big)ds+\int_{0}^{t}{\big( X_s,\Phi(s)X_s \big)_0}ds \\ &+2\int_{0}^{t}\big(X_s,Z_s dW_s\big)_s,
	\end{align*}
	for all $t\in [0,T]$, here {$| \cdot |_{t}$ represents the norm generated by the inner product $(\cdot, \cdot)_t$ and} $H^t$ represents the Hilbert space $H$ equipped with the inner product $(\cdot, \cdot)_t$, {$\Phi(s)$ is the operator defined in \eqref{08201354}.}
\end{lemma}

We remark  that the inner products $\{(\cdot,\cdot)_{0,t}\}_{t \in [0,T]}$ on $L^2_\sigma(\mathcal{O}_0)$, defined in \eqref{njddy}, {satisfy} conditions (\hyperlink{C1}{C1}) and (\hyperlink{C2}{C2}) with $H_{0,\sigma}^1(\mathcal{O}_0)$ replaced by $L_\sigma^2(\mathcal{O}_0)$ and $\Phi(s)$ replaced by $\Phi_0(s)$ (See (\ref{01292005}) for the definition of $\Phi_0(s)$).  As shown in Subsection \ref{section3}, for any solutions $u$ and $v$ of equation \eqref{ns}, their Piola transformations $\tilde{u}$ and $\tilde{v}$ satisfy \eqref{01301253}, with $g_N$ replaced by {the constant} 1 in the definition of $A_2$. Moreover, the equations hold in $L^2_\sigma(\mathcal{O}_0)$. Therefore, by Lemma \ref{ito2}, we have
\begin{align*}
	&|\tilde{u}(t)-\tilde{v}(t)|^2_{0,t}=-2\int_{0}^{t}\big(M(\tilde{u}-\tilde{v})(s),\tilde{u}-\tilde{v}\big)_{0,s}ds+2\int_{0}^{t}\big(L_h(\tilde{u}-\tilde{v})(s),\tilde{u}-\tilde{v}\big)_{0,s} ds\\&\ +2\sum_{k=1}^{\infty}\int_{0}^{t}\big(\widetilde{\sigma}_k(s,\tilde{u})-\widetilde{\sigma}_k(s,\tilde{u}),\tilde{u}-\tilde{v}\big)_{0,s} d\beta^k_s+\sum_{k=1}^{\infty}\int_{0}^{t}|\widetilde{\sigma}_k(s,\tilde{u})-\widetilde{\sigma}_k(s,\tilde{v})|_{0,s}^2ds\\&\ +\int_{0}^{t}\big(\Phi_0(s)(\tilde{u}-\tilde{v}),\tilde{u}-\tilde{v}\big)_{0,0}ds-2\int_{0}^{t}(N(\tilde{u},\tilde{u})-N(\tilde{v},\tilde{v}),\tilde{u}-\tilde{v})_{0,s} ds+|{u}_0-{v}_0|_{0,0}^2.
\end{align*}
Using \eqref{01291451}, \eqref{01291452} and \eqref{01312053} with $g_N$ replaced by 1, we deduce that
\begin{align}\label{uniq}
	&\|u(t)-v(t)\|^2_{L^2(\mathcal{O}_t)}\\
	=&\ \|u_0-v_0\|^2_{L^2(\mathcal{O}_0)}-2\int_{0}^{t}\|\nabla (u-v)(s)\|^2_{L^2(\mathcal{O}_s)}ds-2\int_{0}^{t}\big((u\cdot\nabla)u-(v\cdot\nabla)v,u-v\big)_{L^2(\mathcal{O}_s)}ds	\nonumber\\&\ +2\sum_{k=1}^{\infty}\int_{0}^{t}\big(\sigma_k(s,u(s))-\sigma_k(s,v(s)),u(s)-v(s)\big)_{L^2(\mathcal{O}_s)}d\beta^k_s \nonumber\\&\
	+\sum_{k=1}^{\infty}\int_{0}^{t}\|\sigma_k(s,u(s))-\sigma_k(s,v(s))\|^2_{L^2(\mathcal{O}_s)}ds\nonumber\nonumber\\\leq&\ \|u_0-v_0\|_{L^2(\mathcal{O}_0)}^2-\int_{0}^{t}\|\nabla(u-v)\|_{L^2(\mathcal{O}_s)}^2ds+C\int_{0}^{t}\|v\|^2_{L^4(\mathcal{O}_s)}\|u-v\|_{L^2(\mathcal{O}_s)}^2ds \nonumber\\&\
	+C\int_{0}^{t}f(s)\| u- v\|_{L^2(\mathcal{O}_s)}^2ds +2\sum_{k=1}^{\infty}\int_{0}^{t}\big(\sigma_k(s,u(s))-\sigma_k(s,v(s)),u(s)-v(s)\big)_{L^2(\mathcal{O}_s)}d\beta^k_s. \nonumber\
\end{align}
Applying It\^o's formula to $e^{-C\int_{0}^{t}(f(s)+\|v\|_{L^4(\mathcal{O}_s)}^2)ds}\|u(t)-v(t)\|^2_{L^2(\mathcal{O}_t)}$ and taking expectations, we get for any $t\leq t_0$,
$$E\big[\|u(t)-v(t)\|^2_{L^2(\mathcal{O}_t)}e^{-C\int_{0}^{t}(f(s)+\|v\|^2_{L^4(\mathcal{O}_s)})ds}\big]\leq \|u_0-v_0\|_{L^2(\mathcal{O}_0)}^2, $$
which implies the uniqueness of equation \eqref{ns} {when} $u_0=v_0$.

Next we prove the existence of the solution. By Proposition \ref{proposition4.1}, we know that for any ${N >0}$, the globally modified Navier-Stokes equation \eqref{gns} has a unique solution $u^N$ on $[0,t_0]$. For each ${N >0}$, define a stopping time
$$\tau^N=\text{inf}\{t\geq 0:\|u^N(t)\|_{H^1(\mathcal{O}_t)}>N\}\wedge t_0.$$
Then for any $t\in[0,t_0]$ and $\varphi\in C_{c,\sigma}^2({\mathcal{D}_{t_0}})$, we have
\begin{align*}
	&\nonumber \   \   \  \int_{\mathcal{O}_{t\wedge\tau^N}}u^N(t\wedge\tau^N,x)\cdot\varphi(t\wedge\tau^N,x)dx-\int_{\mathcal{O}_0} u(0,y)\cdot\varphi(0,y)dy  \\ &\nonumber=\int_{0}^{t\wedge\tau^N}\int_{{\mathcal{O}_s}} \Delta u^N(s,x)\cdot\varphi(s,x)dxds+\int_{0}^{t\wedge\tau^N}\int_{{\mathcal{O}_s}}  u^N(s,x)\cdot{\varphi}'(s,x)dxds \\
	&\ \ + \int_{0}^{t\wedge\tau^N}\int_{{\mathcal{O}_s}}(u^N\cdot\nabla \varphi)(s,x) u^N(s,x)dxds+ \sum_{k=1}^{\infty}\int_{0}^{t\wedge\tau^N}\int_{\mathcal{O}_s} \sigma_k(s, u^N(s))\cdot\varphi(s)dxd\beta^k_s.\nonumber
\end{align*}
For $M>N$, replacing $t$ with $t\wedge\tau^N\wedge\tau^M$ and following the same steps as in the proof of uniqueness, we can show that
$$u^N(\cdot\wedge\tau^N\wedge\tau^M)=u^M(\cdot\wedge\tau^N\wedge\tau^M),$$
which implies
$$\tau^N\leq\tau^M, \text{and\ } u^N(\cdot\wedge\tau^N)=u^M(\cdot\wedge\tau^N).$$
We claim that:
\begin{align}\label{02010004}
	{\mathbb{P}-a.s, \text{there exists}\ N_0(\omega) \ \text{such that for} \ N\geq N_0(\omega), \tau^N=\tau^{N_0(\omega)}=t_0.}
\end{align}
Based on \eqref{02010004} and the local compatibility of $\{u^N\}_{{N > 0}}$, we can define a $H_{0,\sigma}^1(\mathcal{O}_\cdot)$-valued predictable random process $u$ such that for any ${N > 0}$,
\begin{align}\label{02010030}
	u(t)\ \dot{=}\ u^N(t), t \in [0, \tau^N].
\end{align}
Moreover, we can infer from \eqref{02010004} that
$$u\in C\big([0,t_0],H_{0,\sigma}^1(\mathcal{O}_\cdot)\big)\cap L^2\big([0,t_0],H^2\cap{H}^1_{0,\sigma}(\mathcal{O}_\cdot)\big), \ \mathbb{P}-a.s.$$
{and} that $u$ is a solution to the equation \eqref{ns} on $[0,t_0]$.

Now, it remains to provide the proof of the claim \eqref{02010004}, which is based on the following estimate inspired by Theorem 2.1 in \cite{KV}.
\begin{lemma}\label{03311625}
	Define the function $\Theta:\mathbb{R}_+\rightarrow\mathbb{R}_+$ by
	$$\Theta(x)=\mathrm{log}\big(1+\mathrm{log}(1+x)\big).$$
	Then there exists a constant $C>0$ independent of $N$ such that
	\begin{align}\label{02010041}
		E\sup_{s\in[0,t_0]}\Theta(\|u^N(s)\|_{H^1(\mathcal{O}_s)}^2)\leq \Theta(\|u_0\|_{H^1(\mathcal{O}_0)}^2)+C(1+\|u_0\|_{L^2(\mathcal{O}_0)}^2).
	\end{align}
\end{lemma}
\begin{proof}
Applying Theorem 2.5 in \cite{PWZZ} to equation \eqref{01301253}, we derive that
	\begin{align*}
		\|\widetilde{u^N}(t)\|_{1,t}^2-\|u_0\|^2_{1,0}=&-2\int_{0}^{t}\int_{\mathcal{O}_0}(P_0h)^{-1}(P_0L^\#_h\widetilde{u^N})(y)\cdot (P_0L^\#_h\widetilde{u^N})(y)dyds\\&+2\int_{0}^{t}\int_{\mathcal{O}_0}(P_0h)^{-1}(P_0h)g_N(\|\widetilde{u^N}\|_{1,s})N(\widetilde{u^N},\widetilde{u^N})(y)\cdot (P_0L^\#_h\widetilde{u^N})(y)dyds\\&+2\int_{0}^{t}\int_{\mathcal{O}_0}(P_0h)^{-1}(P_0h)M\widetilde{u^N}(y)\cdot (P_0L^\#_h\widetilde{u^N})(y)dyds\\&-2\int_{0}^{t}\big(\widetilde{\sigma}(s,\widetilde{u^N}(s))dW_s, P_0L_h^\#\widetilde{u^N}\big)_{0,0}+\int_{0}^{t}\|\widetilde{\sigma}(s,\widetilde{u^N}(s))\|^2_{1,s}ds\\&+\int_{0}^{t}(\Phi(s)\widetilde{u^N}(s),\widetilde{u^N}(s))_{1,0}ds.
	\end{align*}
	Applying It\^o's formula to $\Theta(\|\widetilde{u^N}\|^2_{1,t})$, we then obtain
	\begin{align*}
		&\sup_{s\in[0,t]}\Theta(\|\widetilde{u^N}\|^2_{1,s})+2\int_{0}^{t}\Theta'(\|\widetilde{u^N}\|^2_{1,s})\int_{{\mathcal{O}_0}}(P_0h)^{-1}(P_0L_h^\#\widetilde{u^N})(y)\cdot(P_0L_h^\#\widetilde{u^N})(y) dyds\\\leq&\ \Theta(\|u_0\|^2_{1,0})+2\int_{0}^{t}\Theta'(\|\widetilde{u^N}\|^2_{1,s})g_N(\|\widetilde{u^N}\|_{1,s})\big|\int_{\mathcal{O}_0}(P_0h)^{-1}(P_0h) N(\widetilde{u^N},\widetilde{u^N})\cdot P_0L_h^\#\widetilde{u^N}(y)dy\big|ds\\&\ +2\int_{0}^{t}\Theta'(\|\widetilde{u^N}\|^2_{1,s})\int_{\mathcal{O}_0}\big|(P_0h)^{-1}(P_0h)M\widetilde{u^N}(y)\cdot (P_0L^\#_h\widetilde{u^N})(y)\big|dyds\\&\ +2\sup_{s\in[0,t]}\big|\int_{0}^{s}\Theta'(\|\widetilde{u^N}\|^2_{1,s})\big(\widetilde{\sigma}(s,\widetilde{u^N}(s))dW_s, P_0L_h^\#\widetilde{u^N}\big)_{0,0}\big|+\int_{0}^{t}\Theta'(\|\widetilde{u^N}\|^2_{1,s})\|\widetilde{\sigma}(s,\widetilde{u^N}(s))\|^2_{1,s}ds\\&\ +\int_{0}^{t}\Theta'(\|\widetilde{u^N}\|^2_{1,s})|(\Phi(s)\widetilde{u^N}(s),\widetilde{u^N}(s))_{1,0}|ds
		+2\int_{0}^{t}\big|\Theta''(\|\widetilde{u^N}\|^2_{1,s})\big|\sum_{k=1}^{\infty}(\widetilde{\sigma}_k(s,\widetilde{u^N}(s)),P_0L^\#_h\widetilde{u^N}(s))^2_{0,0}ds.\\ =&\ \Theta(\|u_0\|_{1,0}^2)+\mathrm{I+II+III+IV+V+VI}.
	\end{align*}
	
	Next, we estimate each term {separately}. For term I, we have
	\begin{align*}
		\mathrm{I}&=2\int_{0}^{t}\Theta'(\|\widetilde{u^N}\|^2_{1,s})g_N(\|\widetilde{u^N}\|_{1,s})\big|\int_{\mathcal{O}_0}(P_0h )^{-1}(P_0h) N(\widetilde{u^N},\widetilde{u^N})\cdot P_0L_h^\#\widetilde{u^N}(y)dy\big|ds\\&\leq C\int_{0}^{t}\Theta'(\|\widetilde{u^N}\|^2_{1,s})\|\widetilde{u^N}\|_{H^2(\mathcal{O}_0)}\|\widetilde{u^N}\|_{1,s}\|\widetilde{u^N}\|_{L^\infty} ds	\\&\leq C\int_{0}^{t}\Theta'(\|\widetilde{u^N}\|^2_{1,s})\|\widetilde{u^N}\|_{H^2(\mathcal{O}_0)}\|\widetilde{u^N}\|^2_{1,s}\big(1+\mathrm{log}\big(\frac{c_1\|\widetilde{u^N}\|_{H^2(\mathcal{O}_0)}^2}{\|\widetilde{u^N}\|^2_{1,s}}\big)\big)^\frac{1}{2} ds,
	\end{align*}
	where $c_1$ is a sufficiently large constant, and the third inequality is due to the Br\'ezis-Gallou\"{e}t-Waigner inequality (see (3.15) in \cite{KV}).
	
	For $a,\varepsilon>0,\mu\geq b>0$, it is {shown in} \cite{FMT} that
	\begin{align*}
		a\mu\big(1+\mathrm{log}\frac{\mu^2}{b^2}\big)^\frac{1}{2}\leq \varepsilon\mu^2+\frac{a^2}{\varepsilon}\mathrm{log}\frac{2a}{\varepsilon b}.
	\end{align*}
	Using the above inequality, we further {deduce that}
	\begin{align*}
		\mathrm{I}&\leq \frac{c}{4}\int_{0}^{t}\Theta'(\|\widetilde{u^N}\|_{1,s}^2)\|\widetilde{u^N}\|^2_{H^2(\mathcal{O}_0)}ds+C\int_{0}^{t}\Theta'(\|\widetilde{u^N}\|_{1,s}^2)\|\widetilde{u^N}\|_{1,s}^4\big(1+\mathrm{log} (C\|\widetilde{u^N}\|_{1,s})\big)ds.
	\end{align*}
	For the stochastic term, by \eqref{02062230-2'}, we have
	\begin{align*}
		E[\mathrm{III}]\leq CE\big[\big(\int_{0}^{t}f(s)\Theta'(\|\widetilde{u^N}\|_{1,s}^2)^2(1+\|\widetilde{u^N}\|_{1,s}^2)^2ds\big)^\frac{1}{2}\big].
	\end{align*}
	As for the remaining terms, it is easy to {see} that
	\begin{align*}
		\mathrm{II}&\leq C\int_{0}^{t}\Theta'(\|\widetilde{u^N}\|_{1,s}^2)\|\widetilde{u^N}\|_{1,s}\|\widetilde{u^N}\|_{H^2(\mathcal{O}_0)}\mathrm{d}s\\
		&\leq\frac{c}{4}\int_{0}^{t}\Theta'(\|\widetilde{u^N}\|_{1,s}^2)\|\widetilde{u^N}\|^2_{H^2(\mathcal{O}_0)}ds+C\int_{0}^{t}\Theta'(\|\widetilde{u^N}\|_{1,s}^2)\|\widetilde{u^N}\|^2_{1,s}ds,\\
		\mathrm{IV}&\leq C\int_{0}^{t}f(s)\Theta'(\|\widetilde{u^N}\|_{1,s}^2)\big(\|\widetilde{u^N}\|^2_{1,s}+1\big)ds,\ \mathrm{V}\leq C\int_{0}^{t}\Theta'(\|\widetilde{u^N}\|_{1,s}^2)\|\widetilde{u^N}\|^2_{1,s}ds,\\ \mathrm{VI}&\leq  C\int_{0}^{t}f(s)|\Theta''(\|\widetilde{u^N}\|_{1,s}^2)|\|\widetilde{u^N}\|_{1,s}^2(1+\|\widetilde{u^N}\|_{1,s}^2)ds.
	\end{align*}
	
	Note that
	\begin{align}\label{02011126}
		\Theta'(x)=\frac{1}{(1+x)(1+\mathrm{log}(1+x))}, \quad |\Theta''(x)|\leq \frac{2}{(1+x)^2\big(1+\mathrm{log}(1+x)\big)}.
	\end{align}
	In combination with the above estimates, we {obtain}
	\begin{align*}
		&E\big[\sup_{s\in[0,t]}\Theta(\|\widetilde{u^N}\|^2_{1,s})\big]\\\leq&\  E\big[\Theta(\|u_0\|^2_{1,0})\big]+CE\big[\int_{0}^{t}(f(s)+1)\big(\Theta'(\|\widetilde{u^N}\|_{1,s}^2)+\|\widetilde{u^N}\|_{1,s}^2|\Theta''(\|\widetilde{u^N}\|_{1,s}^2)|\big)\big(1+\|\widetilde{u^N}\|_{1,s}^2\big)ds\big]\\&\ +CE\big[\big(\int_{0}^{t}f(s)\Theta'(\|\widetilde{u^N}\|_{1,s}^2)^2(1+\|\widetilde{u^N}\|_{1,s}^2)^2ds\big)^\frac{1}{2}\big] \\
		&\ +CE\big[\int_{0}^{t}\Theta'(\|\widetilde{u^N}\|_{1,s}^2)\|\widetilde{u^N}\|_{1,s}^4\big(1+\mathrm{log} (C\|\widetilde{u^N}\|_{1,s})\big)ds\big]\\\leq& E\big[\Theta(\|u_0\|^2_{1,0})\big]+C_{f}+CE\big[\int_{0}^{t}\|\widetilde{u^N}\|_{1,s}^2ds\big].
	\end{align*}
	{On the other hand}, as in \eqref{uniq}, we can show that
	\begin{align*}
		E\big[\int_{0}^{t}\|\widetilde{u^N}\|_{1,s}^2ds\big]\leq C_f(\|u_0\|^2_{L^2(\mathcal{O}_0)}+1).
	\end{align*}
	
	We complete the proof of the lemma.
\end{proof}

\vskip 0.2cm

Now, we are ready to complete the proof of the claim \eqref{02010004}. By the moment estimate \eqref{02010041}, we have
$$\mathbb{P}\big(\tau^N< t_0\big)\leq\frac{\Theta(\|u_0\|_{H^1(\mathcal{O}_0)}^2)+C(1+\|u_0\|^2_{L^2(\mathcal{O}_0)})}{\Theta(N)}. $$
Since $\tau^N$ is increasing, we derive that
\begin{align*}
	\mathbb{P}\Big(\mathop{\scalebox{1.5}[1.5]{$\cup$}}_{N=1}^\infty\{ \tau^N = t_0 \}\Big)=1,
\end{align*}
which completes the proof of the claim \eqref{02010004}.
Thus, we establish the well-posedness of the two dimensional Navier-Stokes equation \eqref{ns} on the small interval $[0, t_0]$.

\section{Well-posedness on the whole interval $[0,T]$}\label{section6}

In this section, we are devoted to extending the well-posedness of the stochastic Navier-Stokes equation \eqref{ns} from the small time interval $[0,t_0]$  to the {whole} time interval $[0,T]$. The proof will involve refined constant estimates.

Let us begin by recalling that the choice of time $t_0$ is such that for any $t \in [0,t_0]$,
$$\|(L_h^\#-\Delta_0)v\|_{L^2(\mathcal{O}_0)}\leq\frac{1}{2C_0}\|v\|_{H^2(\mathcal{O}_0)}.$$
where $C_0$ is the smallest constant such that for any $u\in H^2\cap H^1_{0,\sigma}(\mathcal{O}_0)$,
\begin{align}\label{02272030}
	\|u\|_{H^2(\mathcal{O}_0)}\leq C_0\|A_0u\|_{L^2(\mathcal{O}_0)}.
\end{align}
By Proposition \ref{estimate3}, we can find a uniform constant $C_0$ such that if we replace $\mathcal{O}_0$ in \eqref{02272030} with any $\mathcal{O}_t$, $t\in[0,T]$, \eqref{02272030} still holds with $A_0$ replaced with $A_t$. For $t \geq t_0$, the reference domain $\mathcal{O}_{0}$ should be replaced by $\mathcal{O}_{t_0}$, and the corresponding parametrization map $r(t, \cdot)$ and its inverse $\bar{r}(t, \cdot)$ should be replaced by
\begin{align*}
	&r^{t_0}(t, \cdot):\mathcal{O}_{t_0}\rightarrow\mathcal{O}_t,
	\  r^{t_0}(t,z)=r\big(t,\bar{r}(t_0,z)\big),\\&\bar{r}^{t_0}(t, \cdot):\mathcal{O}_{t}\rightarrow\mathcal{O}_{t_0},\ \bar{r}^{t_0}(t,x)=r\big(t_0,\bar{r}(t,x)\big).
\end{align*}
Recalling the definition of $L^\#_h$ in \eqref{02021304}, the corresponding elliptic operator $L^\#_{h_{t_0}}$ is now given by
\begin{align}\label{02021310}
	[{L}^\#_{h_{t_0}}v]_l= h^{t_0}_{kl}(L_{h_{t_0}} v)_k=\frac{\partial r^{t_0}_m}{\partial z_l}\frac{\partial}{\partial z_n}\Big(\frac{\partial}{\partial z_j}\big(v_p\frac{\partial r^{t_0}_m}{\partial z_p}\big) h_{t_0}^{jn}\Big),
\end{align}
where
\begin{align*}
	h_{t_0}^{jk}(t,z)=\frac{\partial\bar{r}_k^{t_0}}{\partial x_i}\frac{\partial\bar{r}_j^{t_0}}{\partial x_i}\big(t,r^{t_0}(t,z)\big), h^{t_0}_{mn}(t,z)=\frac{\partial{r}_l^{t_0}}{\partial z_m}\frac{\partial{r}_l^{t_0}}{\partial z_n}\big(t,z\big).
\end{align*}

We claim that there exists a constant $\delta>0$ independent of $t_0\in[0,T]$, such that for all $t \in [t_0, (t_0+\delta)\wedge T]$,
\begin{align}\label{04171343}
	\|[{L}^\#_{h_{t_0}}v]-\Delta_{t_0}v \|_{L^2(\mathcal{O}_{t_0})}\leq\frac{1}{2C_0}\|v\|_{H^2(\mathcal{O}_{t_0})}.
\end{align}
Express \eqref{02021310} in the following form:
$$[{L}^\#_{h_{t_0}}v]_l(t,z)=\sum_{k,l,|\alpha|\leq 2}P^{t_0}_{\alpha,(k,l)}(t,z)\frac{\partial^{|\alpha|} v_k }{\partial\ z^\alpha}.$$
To prove the claim \eqref{04171343}, it is sufficient to show that there exists a function $\omega(\cdot):\mathbb{R}\rightarrow\mathbb{R}$, which is continuous at 0 with $\omega(0)=0$, such that for arbitrary $\alpha$, $k$, $l$, $t_0\in[0,T]$ and $t\geq t_0$,
\begin{align}\label{02021742}
	\|P^{t_0}_{\alpha, (k,l)}(t,\cdot)-P^{t_0}_{\alpha, (k,l)}(t_0,\cdot)\|_{L^\infty(\mathcal{O}_{t_0})}\leq\omega(|t-t_0|).
\end{align}
In particular, $P^{t_0}_{\alpha, (k,l)}(t_0,\cdot)$ satisfies
\begin{equation}P^{t_0}_{\alpha, (k,l)}(t_0,\cdot)=
	\begin{cases}
		0,\ |\alpha|=2 \ \text{and} \ \alpha_1\neq\alpha_2 \ \text{or} \ |\alpha|<2, \text{or} \ k\neq l,\\
		1,\ |\alpha|=2\ \text{and}\ \alpha_1=\alpha_2,\text{and }k=l.
	\end{cases}
\end{equation}
For the reader's convenience, we provide a proof of the uniform continuity for the coefficients of second order differential operators. The proofs for the remaining terms follow analogously.
Note that for $\alpha=(j,n)$ and $k,l \in \{1,2\}$,
$$ P^{t_0}_{\alpha, (k,l)}(t,z)=h^{t_0}_{kl}h_{t_0}^{jn}(t,z)=\frac{\partial{r}_m^{t_0}}{\partial z_k}\frac{\partial{r}_m^{t_0}}{\partial z_l}(t,z)
\cdot \frac{\partial\bar{r}_n^{t_0}}{\partial x_i}\frac{\partial\bar{r}_j^{t_0}}{\partial x_i}\big(t,r^{t_0}(t,z)\big),$$
where
$$\frac{\partial {r}_i^{t_0}}{\partial z_j}(t,z)=\frac{\partial r_i}{\partial y_k}\big(t,\bar{r}(t_0,z)\big)\frac{\partial \bar{r}_k}{\partial z_j}(t_0,z),\ \frac{\partial \bar{r}^{t_0}_k}{\partial x_l}(t,x)=\frac{\partial r_k}{\partial y_m}(t_0,\bar{r}(t,x))\frac{\partial\bar{r}_m}{\partial x_l}(t,x).	$$
By the regularity assumptions on $r$ and $\bar{r}$, we can find a constant $C>0$ such that for any $t_0\in[0,T]$ and $t \geq t_0$,
\begin{align}\label{02021607}
	\sup_{z\in\mathcal{O}_{t_0},{i,j\leq 2}}|\frac{\partial {r}_i^{t_0}}{\partial z_j}|(t,z)+ \sup_{x\in\mathcal{O}_{t},{k,l\leq 2}}|\frac{\partial \bar{r}^{t_0}_k}{\partial x_l}|(t,x)\leq C.
\end{align}
Also
\begin{align}
	\nonumber\big|\frac{\partial {r}_i^{t_0}}{\partial z_j}(t,z)-\frac{\partial {r}_i^{t_0}}{\partial z_j}(t_0,z)\big|&\leq\big|\Big(\frac{\partial r_i}{\partial y_k}\big(t,\bar{r}(t_0,z)\big)-\frac{\partial r_i}{\partial y_k}\big(t_0,\bar{r}(t_0,z)\big)\Big)\big|\cdot |\frac{\partial \bar{r}_k}{\partial z_j}(t_0,z)|\\&\leq C\omega_1(|t-t_0|),\label{02021730}
\end{align}
where $\omega_1(\cdot)$ denotes the modulus of continuity in time for the matrix $\{\frac{\partial r_i}{\partial y_k}\}_{{i,k\leq 2}}$.
On the other hand, we can also show that
\begin{align*}
	\big|\frac{\partial \bar{r}^{t_0}_k}{\partial x_l}(t,r^{t_0}(t,z))-\frac{\partial \bar{r}^{t_0}_k}{\partial x_l}(t_0,z)\big|&=\big| \frac{\partial r_k}{\partial y_m}(t_0,\bar{r}(t_0,z))\frac{\partial \bar{r}_m}{\partial x_l}(t,r^{t_0}(t,z))-\frac{\partial r_k}{\partial y_m}(t_0,\bar{r}(t_0,z))\frac{\partial \bar{r}_m}{\partial x_l}(t_0,z)\big|\\&\leq C|\frac{\partial \bar{r}_m}{\partial x_l}(t,{r}(t,\bar{r}(t_0,z)))-\frac{\partial \bar{r}_m}{\partial x_l}(t_0,z)\big|.
\end{align*}
Since $$\frac{\partial \bar{r}_m}{\partial x_l}\big(t,{r}(t,\bar{r}(t_0,z))\big)=\big[\frac{\partial r_\cdot}{\partial y_\cdot}(t,\bar{r}(t_0,z))\big]_{ml}^{-1},$$
by \eqref{02021607}, we derive that
\begin{align}\label{02021731}
	\nonumber&|\frac{\partial \bar{r}_m}{\partial x_l}(t,{r}(t,\bar{r}(t_0,z)))-\frac{\partial \bar{r}_m}{\partial x_l}(t_0,z)\big|\\=\nonumber&\big|\big[\frac{\partial r_\cdot}{\partial y_\cdot}(t,\bar{r}(t_0,z))\big]_{ml}^{-1}-\big[\frac{\partial r_\cdot}{\partial y_\cdot}(t_0,\bar{r}(t_0,z))\big]_{ml}^{-1}\big|\\\leq\nonumber&\big|\big[\frac{\partial r_\cdot}{\partial y_{\cdot}}\big]^{-1}(t,\bar{r}(t_0,z))\cdot \big[\frac{\partial r_\cdot}{\partial y_\cdot}(t,\bar{r}(t_0,z))-\frac{\partial r_\cdot}{\partial y_\cdot}(t_0,\bar{r}(t_0,z))\big]\cdot \big[\frac{\partial r_\cdot}{\partial y_\cdot}(t_0,\bar{r}(t_0,z))\big]^{-1}\big|\\\leq& C\omega_1(|t-t_0|).
\end{align}
From the above estimates, we obtain that for $|\alpha|=2$ and $k,l\in \{1,2\}$,
\begin{equation}\label{contniuous1}
	\|P^{t_0}_{\alpha, (k,l)}(t,\cdot)-P^{t_0}_{\alpha, (k,l)}(t_0,\cdot)\|_{L^\infty(\mathcal{O}_{t_0})}\leq C\omega_1(|t-t_0|).
\end{equation}
We can also infer by applying the same method that for $|\alpha|\leq1$ and $k,l\in \{1,2\}$,
\begin{equation}\label{contniuous2}
	\|P^{t_0}_{\alpha, (k,l)}(t,\cdot)-P^{t_0}_{\alpha, (k,l)}(t_0,\cdot)\|_{L^\infty(\mathcal{O}_{t_0})}\leq C\omega_2(|t-t_0|),
\end{equation}
where $\omega_2(\cdot)$ represents the modulus of continuity in time for $\{\frac{\partial r_i}{\partial y_k},\frac{\partial^2 r_i}{\partial y_k \partial y_j}, \frac{\partial^3 r_i}{\partial y_k \partial y_j \partial y_l}\}_{{i,j,k,l\leq 2}}$. By \eqref{contniuous1} and \eqref{contniuous2}, we get \eqref{02021742}.
In particular, it follows from \eqref{02272030} and \eqref{04171343} that there exists a constant $\delta>0$ such that for any $t_0\in[0,T]$, $t\in[t_0,(t_0+\delta)\wedge T]$,
$$\|A_{t_0}^{-1}P_{t_0}{L}^\#_{h_{t_0}}v-v \|_{H^2(\mathcal{O}_{t_0})}\leq\frac{1}{2}\|v\|_{H^2(\mathcal{O}_{t_0})}. $$

Now we come back to the proof of the global existence of solutions to equation \eqref{ns}.

First, combining the results from the previous section and the first part of this section, we know that for any $t_1\in[0,T]$ and  initial value $x_{t_1}\in H_{0,\sigma}^1(\mathcal{O}_{t_1})$, equation \eqref{ns} has a unique solution $u$ on the interval $[t_1,(t_1+\delta)\wedge T]$ such that
$$u\in C\big([t_1,(t_1+\delta)\wedge T], H_{0,\sigma}^1(\mathcal{O}_\cdot)\big)\cap L^2([t_1,(t_1+\delta)\wedge T], H^2\cap H^1_{0,\sigma}(\mathcal{O}_\cdot)). $$
Therefore, by Yamada-Watanabe thereom, for any $t_1\in[0,T]$, there exists a measurable mapping
\begin{align*}
	\Phi_{t_1}: &H_{0,\sigma}^1(\mathcal{O}_{t_1})\times C_0([t_1,(t_1+\delta)\wedge T], \mathbb{R}^\infty)\longrightarrow\\&  C\big([t_1,(t_1+\delta)\wedge T], H_{0,\sigma}^1(\mathcal{O}_\cdot)\big)\cap L^2\big([t_1,(t_1+\delta)\wedge T], H^2\cap H^1_{0,\sigma}(\mathcal{O}_\cdot)\big),
\end{align*}
such that for any complete probability space $(\Omega,\mathcal{F},\{\mathcal{F}_t\}_{t\in[t_1,t_1+\delta]},\mathbb{P})$, any $\mathcal{F}_t$-adapted $l^2$-cylindrical Wiener process $W_\cdot$ on it, and any initial value $x_{t_1}\in H^1_{0,\sigma}(\mathcal{O}_{t_1})$, $\Phi_{t_1}(x_{t_1}, W_\cdot)$ is the unique solution to equation \eqref{ns} on the interval $[t_1,(t_1+\delta)\wedge T]$. Then we can define an $H^1_{0,\sigma}(\mathcal{O}_\cdot)$-valued predictable process $u$ on $[0,T]$ as follows:
\begin{equation}u(t)=
	\begin{cases}
		\Phi_0(x,W_\cdot)(t),\ \ \ t\in[0,\delta],\\
		\Phi_{n\delta}\big(u(n\delta),W^{n}_\cdot\big)(t),\ t\in[n\delta,(n+1)\delta\wedge T],
	\end{cases}
\end{equation}
where for any $n \geq 1$, $W_{\cdot}^{n}$ represents the Wiener process $\{W_{t}-W_{n\delta}\}_{t \in [n\delta,(n+1)\delta\wedge T]}$.
It is easy to see that $u$ is a solution to equation \eqref{ns} on $[0,T]$. The uniqueness of solutions to equation \eqref{ns} follows from a similar argument used in Section \ref{section5}. The global well-posedness of \eqref{ns} then follows.

\section{Appendix}

In this section, we present a result on the equivalence of several norms on the space $H^2\cap H^1_{0,\sigma}(\mathcal{O}_t)$.
\begin{proposition}\label{estimate3}
	Let $A_t$ and $\Delta_t$ denote the Stokes operator and Laplace operator on $\mathcal{O}_t$, respectively. Define the following norms on $H^2\cap H^1_{0,\sigma}(\mathcal{O}_t)$:
	\begin{itemize} {
		}    \item[(i)]: $\|u\|^2_{H^2(\mathcal{O}_t)}=\|u\|_{L^2(\mathcal{O}_t)}^2+\|\nabla u\|^2_{L^2(\mathcal{O}_t)}+\|D^2u\|^2_{L^2(\mathcal{O}_t)}$,
		\item[(ii)]: $\|u\|^2_2=\|u\|_{L^2(\mathcal{O}_t)}^2+\|\Delta_t u\|_{L^2(\mathcal{O}_t)}^2$,
		\item[(iii)]: $\|u\|^2_3=\|u\|_{L^2(\mathcal{O}_t)}^2+\|A_tu\|^2_{L^2(\mathcal{O}_t)}$,
		\item[(iv)]: $\|u\|^2_4=\|\Delta_tu\|^2_{L^2(\mathcal{O}_t)}$,
		\item[(v)]: $\|u\|^2_5=\|A_tu\|^2_{L^2(\mathcal{O}_t)}$.
	\end{itemize}
	Then the above five norms are {all} equivalent, and the equivalence constants are independent of $t \in [0,T]$.
\end{proposition}
\begin{proof}
	Obviously, $\|u\|_3\leq \|u\|_2\leq \|u\|_{H^2(\mathcal{O}_t)}$. By Lemma 1 in \cite{H}, there exists a constant $C>0$ such that for any $t\in[0,T]$, $u\in H^2\cap H_{0,\sigma}^1(\mathcal{O}_t)$
	\begin{align*}
		\|u\|_{H^2(\mathcal{O}_t)}\leq C(\|A_tu\|_{L^2(\mathcal{O}_t)}+\|\nabla u\|_{L^2(\mathcal{O}_t)}).
	\end{align*}
	Since
	\begin{align*}
		\|\nabla u\|_{L^2(\mathcal{O}_t)}^2=-\int_{\mathcal{O}_t}u \cdot A_tu\ dx\leq \frac{1}{2}\|A_tu\|_{L^2(\mathcal{O}_t)}^2+\frac{1}{2}\|u\|_{L^2(\mathcal{O}_t)}^2,
	\end{align*}
	we can deduce that
	\begin{align*}
		\|u\|_{H^2(\mathcal{O}_t)}\leq C(\|A_tu\|_{L^2(\mathcal{O}_t)}+\|u\|_{L^2(\mathcal{O}_t)}),
	\end{align*}
	which implies $\|u\|_{H^2(\mathcal{O}_t)}\leq \|u\|_3$, proving the equivalence of $\|\cdot\|_{H^2(\mathcal{O}_t)},\|\cdot\|_2$ and $\|\cdot\|_3$.
	For the equivalence between $\|\cdot\|_2$ and $\|\cdot\|_4$, by the Rayleigh-Faber-Krahn inequality, there exists a constant $c_0>0$ independent of $t\in[0,T]$ such that for any $u\in H^2\cap H^1_0(\mathcal{O}_t)$,
	$$\|\Delta_t u\|_{L^2(\mathcal{O}_t)} \geq c_0\|u\|_{L^2(\mathcal{O}_t)},$$
	implying the equivalence of $\|\cdot\|_2$ and $\|\cdot\|_4$. By the result in \cite{K}, we know that the spectral gap of $A_t$ is strictly greater than the spectral gap of $\Delta_t$, which implies that the norms $\| \cdot \|_3$ and $\| \cdot \|_5$ are equivalent, and the equivalence constant is independent of $t$.

\end{proof}
\vskip 0.3cm

\noindent{\bf Acknowledgements}\\
This work is partially supported by National Key R\&D
Program of China(No.2022YFA1006001), and by  NSFC (No. 12131019, 12171321, 12371151).

\end{document}